\documentclass[12pt]{amsart}

\usepackage{amsfonts,latexsym,amsthm,amssymb,amsmath,amscd,euscript,bibentry}
\usepackage{tikz-cd}
\usepackage[margin = 3cm]{geometry}
\usepackage[shortlabels]{enumitem}
\usepackage[utf8]{inputenc}
\usepackage{hyperref}
\usepackage{arydshln}
\usepackage{dirtytalk}
\usepackage{float}
\usepackage{multirow}
\usepackage[normalem]{ulem}
\hypersetup{
    colorlinks=true,
    linkcolor=red,
    filecolor=red,      
    urlcolor=red,
}
 
\definecolor{purple}{rgb}{0.59, 0.44, 0.84}

\renewcommand{\bar}{\overline}
\newtheorem{theorem}{Theorem}[section]
\newtheorem{proposition}[theorem]{Proposition}
\newtheorem{Lemma}[theorem]{Lemma}
\newtheorem{corollary}[theorem]{Corollary}

\newtheorem{Conjecture}[theorem]{Conjecture}

\theoremstyle{definition}
\newtheorem{definition}[theorem]{Definition}

\theoremstyle{remark}
\newtheorem{remark}{Remark}

\usepackage{fancyhdr}
\usepackage{caption} 
\usepackage{thmtools}
\usepackage[framemethod=TikZ]{mdframed}
	\mdfdefinestyle{mdrecbox}
		{
			linewidth=0.5pt,
			skipabove=12pt,
			frametitleaboveskip=5pt,
			frametitlebelowskip=0pt,
			skipbelow=2pt,
			frametitlefont=\bfseries,
			innertopmargin=4pt,
			innerbottommargin=8pt,
			nobreak=true,
		}
	\declaretheoremstyle
		[
			headfont=\bfseries,
			mdframed={style=mdrecbox},
			headpunct={\\[3pt]},
			postheadspace={0pt},
		]
		{thmrecbox}
	\declaretheorem[style=thmrecbox,name=Example, numberlike=theorem]{examplebox}

\newcommand{\nc}{\newcommand}

\nc{\R}{\mathbb R}
\nc{\C}{\mathbb C}
\nc{\F}{\mathbb F}
\nc{\Q}{\mathbb Q}
\nc{\Z}{\mathbb Z}

\nc{\N}{\mathbb N}
\nc{\B}{\mathbb B}
\nc\scalemath[2]{\scalebox{#1}{\mbox{\ensuremath{\displaystyle #2}}}}
\nc{\limin}{\underline{\lim}}
\nc{\limsu}{\overline{\lim}}
\nc{\bl}{\color{blue}}
\newtheorem*{theorem*}{Theorem}

\nc{\cT}{\mathcal T}
\nc{\cP}{\mathcal P}
\nc{\cM}{\mathcal M}
\nc{\cC}{\mathcal C}
\nc{\cB}{\mathcal B}
\nc{\cO}{\mathcal O}
\nc{\quat}{\left(\frac{a,b}{F}\right)}
\nc{\cS}{\mathcal S}
\nc{\e}{\mathbf{e}}
\nc{\w}{\mathbf{w}}

\nc{\f}{\mathbf{f}}
\nc{\s}{\text{ }}
\nc{\Mod}{\operatorname{Mod}}
\nc{\Aut}{\operatorname{Aut}}
\nc{\del}{\partial}
\nc{\inter}{\mathrm{o}}
\nc{\close}[1]{\overline{#1}}
\nc{\pderiv}[2]{\frac{\partial #1}{\partial #2}}
\nc{\tr}{\operatorname{tr}}
\nc{\disc}{\text{disc}}

\catcode`,\active

\catcode`\,12

\newcommand*\HYPERskip{&}
\catcode`,\active
\newcommand*\pfq{
\begingroup
\catcode`\,\active
\def ,{\HYPERskip}
\doHyper
}
\catcode`\,12
\def\doHyper#1#2#3#4#5{
\, _{#1}F_{#2}\left[\begin{matrix}#3 \smallskip \\  #4\end{matrix} \; ; \; #5\right]
\endgroup
}

\catcode`,\active

\catcode`\,12

\renewcommand{\epsilon}{\varepsilon}

\title[Hypergeometric Series and $L$-values of Modular Forms]{$L$-values of certain weight 3 Modular Forms and Transformations of  Hypergeometric Series}
\author{Esme Rosen }
\date{}

\begin{document}

\begin{abstract}
Recently, Allen, Grove, Long, and Tu proposed an explicit Hypergeometric-Modularity method which gives a concrete link between certain hypergeometric objects and modular forms. The theory is exemplified by a collection of  199 weight 3 modular forms. Among other properties their process shows that the  $L$-value  of such a modular form at 1 is an explicit multiple of a ${}_3F_2(1)$ hypergeometric series. Using the framework of a finite Coxeter group governing the invariance group of normalized ${}_3F_2(1)$ series, this paper fully classifies and describes the possible Hecke eigenforms whose $L$-values that can be obtained using this method. In addition, we determine when these modular forms differ by twist of a finite-order character using the perspective of hypergeometric functions.  As one application, we reinterpret a classical identity of hypergeometric series as a formula involving $L$-values of two Hecke eigenforms that differ by a twist. \end{abstract}
\maketitle

\tableofcontents

\section{Introduction}

Hypergeometric functions are classical complex-analytic functions studied by Gauss, Kummer, and many others in the 1800's. However, since the 1980's, there has been a surge of interest in the place of hypergeometric functions in arithmetic geometry, beginning with the work of \cite{beukersheckman}, \cite{greene}, \cite{katz}, \cite{stiller}, \cite{wolfart}, and many others. Modular forms are another object in arithmetic geometry most famously used in the proof of Fermat's last theorem, which are also of wide interest in arithmetic geometry and number theory. While the relation of modular forms on arithmetic triangle groups and certain ${}_2F_1(z)$ hypergeometric functions goes back to Fricke and Klein \cite{frickeklein}, the construction of Galois representations attached to hypergeometric series by Katz \cite{katz} led to increased attention to the relationship between hypergeometric functions and modular forms---refer to \cite{dawseymccarthy} for a survey. 

Recently, Allen, Grove, Long, and Tu \cite{aglt} studied the relation between certain ${}_3F_2(1)$ hypergeometric functions and modular forms of weight 3, as a part of the program introduced by the final two authors and their collaborators in \cite{flrst}, \cite{lilongtu}, and \cite{whipple}. One of the central examples in \cite{aglt} concerns the hypergeometric datum $$HD(r,s):=\{\{1/2,1/2,r\},\{1,1,s\},\{1\}\}$$ for certain rational numbers $r,s$, and modular forms. Among other results, they showed that a normalized---see \eqref{lval}--- classical hypergeometric series ${}_3F_2(HD(r,s))$ is the exact $L$-value at 1 of an explicit weight 3 modular form denoted $\mathbb{K}_2(r,s)$ to be defined below in (\ref{eq:K2}). There are 199 cases where this modular form is holomorphic and congruence. The primary goal of this paper is to derive properties of these modular forms and their $L$-values using the perspective of ${}_3F_2(1)$ hypergeometric series and their transformations. See for example Theorem \ref{twistingpre} below, which relates the classical Pfaff formula to a twisting property of the $\mathbb{K}_2(r,s)$. Another important tool is the theory of Coxeter groups and transformations of ${}_3F_2(1)$ series. There is a long history of studying the invariance groups of normalized hypergeometric series evaluated at 1 using finite Coxeter groups---the basic idea was known to Thomae \cite{thomae}; see also \cite{coxeter}, \cite{formichella}, and \cite{green} for a modern explanation.  
We construct a Coxeter group is $$G=\langle A, K\rangle\cong D_6$$ where $A$ arises from Atkin--Lehner involution acting on the underlying modular forms, and $K$ is from \eqref{kummer} below. This enables the classification of the $\mathbb{K}_2(r,s)$ in Theorem \ref{classs}, as well as other applications.

Number theorists are especially interested in modular forms which are \textit{Hecke eigenforms}, since they have Galois representations and other helpful properties. However, the modular form $\mathbb{K}_2(r,s)$ is never a Hecke eigenform unless $(r,s)=(1/2,1)$, which is a well-studied example---see e.g. \cite{ahlgrenono}, \cite{osburnstraub}. To construct Hecke eigenforms, Allen et al. introduced several ways to classify the $\mathbb{K}_2(r,s)$ functions. Essentially, the idea is to use the \say{Galois conjugates} of $\mathbb{K}_2(r,s)$, with respect to $\Q(\zeta_M)$, where $M$ is the lowest common denominator of $1/2,r,s$ and $\zeta_M$ denotes a primitive $M$th root of unity. To be more precise, consider $\mathbb{K}_2(1/8,1)$. Then $M=8$, and the conjugates are to $\mathbb{K}_2(i/8,1)$, where $i= 1,3,5,7$, corresponding to $ \zeta_8,$ $\zeta_8^3$, $\zeta_8^5$, and $\zeta_8^7$. This situation is described as \textit{Galois} in \cite{aglt}. A definition is given in Definition \ref{galois}. This concept is very effective for constructing weight 2 Hecke eigenforms (see \cite{rosen}), but breaks down somewhat in weight 3, for geometric reasons.  We use the Coxeter group interpretation to avoid this problem here. The conjugates in this setting correspond to the commutator of the symmetry group $G$ for the examples in this paper. The more general definition of conjugates given in \cite{aglt} is different, and stated in Definition \ref{conj}. We briefly explain the geometric origin of this definition and its limitations in Appendix II, but the problematic cases are addressed primarily in \cite{rosen2}.

For $r=1/b$ and $s=s_1$, Galois families can be written as
\begin{equation}\label{srs}
    \mathfrak{S}_{r,s}:=\{\mathbb{K}_2(i/b,s_i)\}_{i\in (\Z/b\Z)^\times}
\end{equation} similar to the example above. 
 Our first main theorem illustrates the relationship of these families with the Hecke operators.
\begin{theorem}\label{completepre}
   Assume $\mathfrak{S}_{r,s}$ is a Galois family having the form given in \eqref{srs}. Let $T_p$ be the Hecke operator for $p$ a prime congruent to $i$ modulo $M$. For any $j\in (\Z/b\Z)^\times$, let $k=i\cdot j$ as elements of $(\Z/b\Z)^\times$. Then there are non-zero integers $C(p,j)$ so that $$T_p\, \mathbb{K}_2(j/b,s_j)=C(p,j)\cdot \mathbb{K}_2(k/b,s_k).$$ Moreover, for explicit $\beta_i\in \bar{\Q}$, there is a Hecke eigenform $f_{r,s}$ so that $$f_{r,s}=\sum_{i\in (\Z/b\Z)^\times}\beta_i\cdot \mathbb{K}_2(i/b,s_i).$$ Up to quadratic twisting, there are exactly 17 Hecke eigenforms that occur.
\end{theorem}
Theorem \ref{completepre} is mostly proved in Lemma \ref{hecke2} and Theorem \ref{complete}. The final line is proved in Section \ref{class}. Due to the quadratic twists, there is some ambiguity in the meaning of $f_{r,s}$. We fix a twist for this modular form in Table \ref{tab:my_label}, and assume $f_{r,s}$ refers to this throughout. We cross-reference our Hecke eigenforms with the labels given in the $L$-functions and Modular Forms Database (LMDFB, \cite{lmfdb}) in Table \ref{tab}; but also see Remark \ref{lmfd} below. A Hecke eigenform $f$ has \textit{complex multiplication} by an imaginary quadratic field $K$ if the Fourier coefficient $a_n(f)$ of $f$ is zero whenever $(n)$ is inert in $K$.  We focus exclusively on the situation where $f_{r,s}$ is non-CM, which accounts for 12 of the cases given in Theorem \ref{completepre}. The reason is two of the CM $f_{r,s}$ appearing are already well studied in \cite{ahlgrenono},
\cite{aglt}, and \cite{osburnstraub}. The remaining three behave slightly differently and are addressed in \cite{rosen2}, as well as the non-Galois cases, which display an interesting \say{mixed} CM structure. 
 
Because of the work of \cite{aglt} on $L$-values of $\mathbb{K}_2(r,s)$, we have the following corollary.
\begin{corollary}\label{lvalue}
    For any Galois family, we have $$L(f_{r,s},1)=\sum_{i\in (\Z/M\Z)^\times}\beta_i\cdot F(i/b,s_i),$$ where $F(r,s)$ is a normalized 
    ${}_3F_2(1)$ series defined below  in \eqref{frs}.
\end{corollary}

As an application of Corollary \ref{lvalue}, we can study relations between $L$-values using known theorems regarding hypergeometric functions. One such formula due to Kummer (1836) states that 
   \begin{equation}\label{kummer}
       \pfq{3}{2}{1/2,1/2,r}{,1,s}{1}=\frac{\Gamma(s)\Gamma(s-r)}{\Gamma(s-1/2)\Gamma(s-r+1/2)}\pfq{3}{2}{1/2,1/2,1-r}{,1,1/2-r+s}{1}.
   \end{equation} 
The version here is a special case of the more general theorem; see Andrews, Askey, and Roy \cite{aar} Cor. 3.3.5. In the Coxeter group language, we call this transformation $K$, and it is the second generator of $G$. There are numerous ways such identities can be interpreted, especially when modular forms are in the background. Some motivating examples are \cite{chenchu}, \cite{whipple},  and \cite{zudilin2}. 
In our setting, Kummer's formula naturally  becomes a transformation of $L$-values: \begin{equation}\label{relation}
L(\mathbb{K}_2(r,s),1)=\alpha_{r,s}\cdot L(\mathbb{K}_2(1-r,1/2-r+s),1).
\end{equation} Refer to \eqref{lval} below for an explanation. The value $\alpha_{r,s}$ is  explicit but not necessarily algebraic, and depends heavily on $r$ and $s$.

Remarkably, \eqref{relation} \say{extends} to a relation between the Hecke eigenforms $L(f_{r,s},1)$ and $L(f_{1-r,1/2-r+s},1)$ in some cases. Although the $L$-value at 1 of $f_{r,s}$ is a linear combination of several hypergeometric series by Corollary \ref{lvalue}, the constant $\alpha_{r,s}$ in \eqref{relation} is not the same for different choices of conjugate $(r,s)$. For instance, $\alpha_{1/8,5/8}=8\sin(\pi/8)$, while its conjugate has $\alpha_{3/8,7/8}=2\cos(\pi/8)$. Despite that, we obtain the following formula relating the corresponding Hecke eigenforms.
\begin{equation}\label{L2} 
    L(f_{1/8,5/8},1)=-\frac{1}{2}\zeta_{48}\cdot L(f_{1/8,1},1).
\end{equation}
The proof requires careful manipulation of the hypergeometric series using three-term identities. These three-term identities are related in \cite{coxeter} directly to the Coxeter group, which in our case is $G$ as mentioned above. We illustrate the connection of these identities to modular forms by providing a new proof purely using Shimura's theory of periods for modular forms. In this way, certain three-term identities correspond to identities of $L$-values like the one given above. The method of proof is dramatically different than the classical analytic proofs given in e.g. \cite{thomae} of similar identities. \begin{theorem}\label{mainpar}
   For $(r,s)=(1/8,5/8),(1/8,7/8)$, or $(1/12,11/12)$ and $M_1=AKA$ and $M_2=KAK$, there are algebraic constants $\alpha_1,\alpha_2$ so that \begin{equation}\label{eq:3-term}
        F(r,s)=\alpha_1 \cdot F(M_1(r,s))+\alpha_2\cdot F(M_2(r,s)).
    \end{equation}
\end{theorem}
\noindent The proof of this Theorem is given in Lemmas \ref{l1}, \ref{l2}, and \ref{l3}.  Refer to Section \ref{coxe} for more details.

To interpret \eqref{L2} theoretically, consider the modular form $\mathbb{K}_2(1/4,1)$ as an example. The Kummer transformation \eqref{kummer} maps the $L$-value of this modular form to the $L$-value of $\mathbb{K}_2(3/4,5/4)$. A priori these modular forms are unrelated. However, we note $\mathbb{K}_2(3/4,5/4)$ is conjugate to $\mathbb{K}_2(1/4,3/4)$. Refer to Figure \ref{fig:placeholder} for a convenient summary of how these pairs are related. If we write $q=e^{2\pi i\tau}$ for the local uniformizer of a modular form at $i\infty$, then, we observe a striking relationship between  $\mathbb{K}_2(1/4,1)$ and $\mathbb{K}_2(1/4,3/4)$. \begin{align*}
     \mathbb{K}_2(1/4,1)=q - 2 q^5 - 7 q^9 + 14 q^{13} + 18 q^{17} - 32 q^{21} - 21 q^{25} + 14 q^{29}+O(q^{30})\\
    \mathbb{K}_2 (1/4,3/4)=q + 2 q^5 - 7 q^9 - 14 q^{13} + 18 q^{17} + 32 q^{21} - 21 q^{25} - 14 q^{29}+O(q^{30}).
 \end{align*}  A similar property holds for $\mathbb{K}_2(3/4,1)$ and $\mathbb{K}_2(3/4,5/4)$, which enables us to write $f_{1/4,1}=\phi\otimes f_{1/4,3/4}$, where $\phi$ is a quadratic character. This idea extends to all $\mathbb{K}_2(r,s)$ functions.
\begin{theorem}\label{twistingpre}
    Let $h=h(r,s)=r-s +3/2$. If $0<h(r,s)<3/2$, then for any positive integer $n$, the $n$th Fourier coefficients of $\mathbb{K}_2(r,s)(\tau)$ and $\mathbb{K}_2(r,h)(\tau)$ are the same up to a sign. Moreover, if $(r,s)$ belongs to a Galois family, then the completed Hecke eigenforms $f_{r,s}=\chi_{r,h}\otimes f_{r,h}$, where $\chi_{r,s}$ is a finite order character.
\end{theorem}
The proof relies on symmetries of hypergeometric functions, specifcally the Pfaff formula mentioned earlier. Refer to Theorem \ref{twisting} for a precise description of which Fourier coefficients switch sign. The transformation $(r,s)\mapsto (r,r-s+3/2)$ corresponds to the element $AKAKA$ in $G$, connecting the classical theory of ${}_3F_2(1)$ hypergeometric series and Coxeter groups to this twisting result for modular forms.

In the recurring example from \eqref{L2}, we have  $f_{1/8,1}=f_{1/8,5/8}\otimes \chi_{1/8,5/8}$. Note $\chi_{1/8,5/8}$ is a character of order 4, and so $\chi_{r,s}$ is not necessarily quadratic---this happens because the $\beta_i$ are different for each eigenform.
 To explain \eqref{L2}, observe that $\chi_{1/8,5/8}$ is an even character, and so a theorem of Shimura \cite{shimurazeta} implies $L(f_{1/8,5/8}\otimes\chi_{r,s},1)/L(f_{1/8,5/8},1)\in \bar{\Q}$. We state Shimura's result as Theorem \ref{shimura} below.  Our \eqref{L2} is rewritten  as \begin{equation}\label{L}
    L(f_{1/8,5/8},1)=-\frac{1}{2}\zeta_{48}\cdot L(f_{1/8,5/8}\otimes\chi_{1/8,5/8},1). 
\end{equation}  Therefore, using symmetries of hypergeometric series, we are able to obtain the exact value of the quotient from Shimura's theorem. Here is the list of formulas we obtain, excluding the cases when $\chi_{r,s}$ is a quadratic twist.

\begin{theorem}\label{relations}
    For the embeddings of the Fourier coefficients specified in Table \ref{tab:my_label}, we have the relations
    \begin{align*}
       & L(f_{1/8,5/8},1)=-\frac{1}{2}\zeta_{48}\cdot L(f_{1/8,1},1)\quad ,\quad L(f_{1/12,2/3},1)=\sqrt{3}\zeta_8^7\cdot L(f_{1/12,11/12},1)\\&
       L(f_{1/8,3/4},1)=\sqrt{2 - \sqrt{2} - 4 i \sqrt{3 - 2 \sqrt{2}}}\cdot L(f_{1/8,7/8},1).
    \end{align*}

\end{theorem}
 
The paper is organized in two main parts. First, we discuss the necessary background in Section \ref{prelim}. Part I then focuses on properties of the $\mathbb{K}_2(r,s)(\tau)$. The first half of this Part is dedicated to proving Theorem \ref{completepre}. Then, we prove Theorem \ref{twistingpre} and and use that to classify all possible Galois families.  
In Part II, we apply the results of the first half to prove the three-term identities in Theorem \ref{mainpar} using modular forms. We then prove the extension of the Kummer transformation to the $L$-values of Hecke eigenforms Theorem \ref{relations}. After discussing some further directions, we provide a table of all Hecke eigenforms obtained using our method in Appendix I. In Appendix II, we provide some geometric explanation of conjugates.

\bigskip

\subsection*{Acknowledgments}  The author would like to express her gratitude to Ling Long and Fang-Ting Tu for suggesting this problem based on their recent work, and for their helpful guidance and suggestions along the way. She would also like to thank John Voight, Yifan Yang, and Wadim Zudilin for helpful discussions. Finally, she thanks Michael Allen, Brian Grove, and Hasan Saad for their helpful comments on an earlier version of this paper.

The author was supported by a summer research assistantship from the Louisiana State University Department of Mathematics in 2024.

\bigskip

\section{Preliminaries}\label{prelim}

In this section, we introduce the $\mathbb{K}_2(r,s)$ functions and recall some of their important properties from \cite{aglt}. We also briefly explain how we can frame our discussion throughout the paper using a finite Coxeter group.
\subsection{The $\mathbb{K}_2(r,s)$ Functions}
For rational parameters $a_i$ and $b_i$ and $\lambda\in \C\setminus\{0,1\}$, let $(a)_k=a(a+1)\dots(a+k-1)$. Then define the generalized hypergeometric series by $$\pfq{n}{n-1}{a_1,a_2,\dots,a_n}{b_1,b_2,\dots,b_{n-1}}{\lambda}=\sum_{k=0}^\infty\frac{(a_1)_k(a_2)_k\dots(a_n)_k}{(b_1)_k(b_2)_k\dots(b_{n-1})_k}\frac{\lambda^k}{k!}.$$ There is also an integral representation:  $$ \int_0^1x^{r-1}(1-x)^{s-r-1}\pfq{2}{1}{a,b}{,c}{x}dx=\frac{1}{B(r,s-r)}\pfq{3}{2}{a,b,r}{,c,s}{1}.$$ The basic idea of Allen et al.'s \cite{aglt} construction  is to use Ramanujan's theory of elliptic functions to alternate bases to identify the integrand with a modular form of weight 3. One of the richest examples of this process arises from $$\pfq{2}{1}{1/2,1/2}{,1}{\lambda}$$ where $\lambda$ is the modular lambda function, which is studied intensively by Allen et al.  It is well known---see e.g. \cite{yang}---that this hypergeometric function is an elliptic function and a weight 1 modular form on the congruence group $\Gamma(2)$. Some other choices of ${}_2F_1(t)$ are noted in \cite{aglt}, but the choice of datum are restrictive, since we need the hypergeometric function to be a modular form. See also \cite{grove} for the choice of datum $\{\{1/2,2/3\},\{1,1\}\}$. The eta product form of $\lambda$ and of elliptic functions allows us to write the integrand as an eta quotient, 
\begin{equation}\label{eq:K2}
    \mathbb{K}_2(r,s)(\tau)= \frac{\eta(\tau/2)^{16s-8r-12}\eta(2\tau)^{8r+8s-12}}{\eta(\tau)^{24s-30}},
\end{equation} where $\eta(\tau)$ is the Dedekind eta function. Also, let $q=e^{2\pi i\tau}$ as before and define \begin{equation}\label{N}
    N(r)=\frac{48}{\gcd(24r,24)}.
\end{equation} 

\begin{Lemma}[\cite{aglt}]\label{aglt}
      The function $\mathbb{K}_2(r,s)$ is a weight 3 cusp form for some finite index subgroup $\Gamma$ of $SL_2(\Z)$ if and only if $0<r<s<3/2$, and $$\mathbb{K}_2(r,s)(\tau)=2^{1-4r}\lambda^r(1-\lambda)^{s-r-1}\pfq{2}{1}{1/2,1/2}{,1}{\lambda}q\frac{d}{dq}\log \lambda.$$ Moreover, $\Gamma$ is congruence if and only if the exponents in the eta product  \eqref{eq:K2} are integral, and the level is $\mathcal{N}:=N(r)N(s-r).$
\end{Lemma}

They also showed there are finitely many $(r,s)$ which correspond to holomorphic  congruence cusp forms, and the set of all of these pairs is \begin{equation}\label{s2}
    \mathbb{S}_2':=\{(r,s)\s|\s 0<r<s<3/2, \s 24s, 8(r+s)\in \Z\}.
\end{equation} There are exactly 199 such pairs.\footnote{There are 193 pairs that are non-degenerate, which corresponds to the set $\mathbb S_2$ given in \cite{aglt}. It was claimed in \cite{aglt} that there are 167 non-degenerate pairs, but this assertion is incorrect.} \begin{remark}
    The $q$-expansion of $\mathbb{K}_2(r,s)(\tau)$ should a priori have local uniformizer  $e^{2\pi i\tau/N}$. However, we can lift all of our congruence $\mathbb{K}_2(r,s)(\tau)$ to $\Gamma_1(\mathcal{N})$ by mapping $\tau\mapsto N\tau$, and for $\mathbb{K}_2(r,s)(N\tau)$ the local uniformizer is $q=e^{2\pi i\tau}$ as above. We simply write $N$ if $r$ is clear. Note that by construction, if $r=i/b$ where $i$ and $b$ are coprime, then $N=2b$. By a slight abuse of notation, when we write $\mathbb{K}_2(r,s)$ without a variable, we mean $\mathbb{K}_2(r,s)(N\tau)$.
\end{remark}

The integral representation of hypergeometric functions combined with Lemma \ref{aglt} gives that \begin{align}
    \int_0^1\lambda^r(1-\lambda)^{s-r-1}\pfq{2}{1}{1/2,1/2}{\hspace{.8cm}1}{\lambda}\frac{d\lambda}{\lambda}=-2^{4r}\pi i\int_{0}^{i\infty} \mathbb{K}_2(r,s)(\tau)d\tau\\\nonumber=B(r,s-r)\cdot\pfq{3}{2}{1/2\hspace{.3cm}1/2\hspace{.3cm}r}{ \hspace{1.1cm}1\hspace{.5cm}s}{1}.
\end{align} Here $$B(a,b)=\int_0^1t^{a-1}(1-t)^{b-1}dt$$ is the beta function.
      We also have the integral formula for the special $L$-values attached to $\mathbb{K}_2(r,s)$, 
\begin{align}\label{lval}
    L(\mathbb{K}_2(r,s)(N\tau),1)&=-2\pi i\int_{0}^{i\infty} \mathbb{K}_2(r,s)(N\tau)d\tau\\\nonumber&=\frac{2^{1-4r}B(r,s-r)}{N}\cdot\pfq{3}{2}{1/2\hspace{.3cm}1/2\hspace{.3cm}r}{ \hspace{1.1cm}1\hspace{.5cm}s}{1}.
\end{align} For convenience, we establish the notation  \begin{equation}\label{frs}
    F(r,s):=\frac{2^{1-4r}B(r,s-r)}{N}\cdot\pfq{3}{2}{1/2,1/2,r}{,1,s}{1}.
\end{equation}

\begin{remark}
    Our normalization is non-standard and is chosen so that the computations with $L$-values are more convenient. It differs from the $P(r,s)$ function used in \cite{aglt} as follows: $$\pi F(r,s)=N2^{1-4r}P(r,s).$$ 
\end{remark}

Note $\pi F(r,s)$ is a period on a hypergeometric surface, which is important from a geometric point of view. If $F(r,s)$ and $F(r',s')$ are \textit{conjugate}, then they can be defined as periods on the same surfaces---see Lemma \ref{geom} in Appendix II. This reduces to the following definition. As above, $M$ is the lowest common denominator of $1/2,r,s$, where we assume $r$ and $s$ are simplified into lowest terms.
 \begin{definition}\label{conj}
   Two pairs $(r,s)$ and $(r',s')$ are \textit{conjugate}  if there exists an integer $c$ coprime to $M$ so that $r-cr'$ and $s-cs'$ are both integers. 
 \end{definition}

A special type of conjugate family as mentioned in the introduction are \textit{Galois}.

\begin{definition}\label{galois}
  Assume $\{(r_i,s_i)\}_i$ is a family of conjugates. If for $i\neq j$, $r_i\neq r_j$, then we call a the family \textit{Galois}.
\end{definition}

This definition is a slight refinement of the definition given in \cite{aglt} and \cite{grove}, but the idea is the same. As noted in the introduction, in the Galois case, we can write our family as $$\mathfrak S_{i/b,s_i}:=\{\mathbb{K}_2(i/b,s_i)\}_{i\in (\Z/b\Z)^\times}.$$  We use this extensively in Section \ref{building}.

\subsection{Coxeter Group Interpretation} \label{coxe}

There are many classical hypergeometric transformations in terms of symmetries of a finite Coxeter group going back to at least Bailey \cite{bailey} and Thomae \cite{thomae}. A finite Coxeter group is a special kind of finite group, which for us will typically be $S_n$, the symmetric group on $n$ variables, or $D_n$, the dihedral group of order $2n$. For ${}_3F_2(1)$ series, this has been studied in great depth by Beyer et al. \cite{coxeter}, and there are similar results for certain ${}_4F_3(1)$ \cite{formichella}, \cite{green} as well. Using the finite symmetry group is very helpful for our classification of $\mathbb{K}_2(r,s)$ families and for stating other identities. This is because the Coxeter group provides a uniform way to state identities that come from the hypergeometric side (such as analytic continuation formulas) and the modular forms side (such as the Atkin--Lehner involution). 

Based on the work of Bailey, Beyer et al. show that for ${}_3F_2(1)$ series, the invariance group for a properly normalized series denoted  $\hat{F}(r,s)$ is $S_5$, and they use this to classify all two term identities between the hypergeometric series, such as the Kummer transformation. The action of $S_5$ on a hypergeometric multiset $\{\{a,b,c\},\{d,e\}\}$ is expressed by writing the multiset as a column vector $(a,b,c,d,e)^t$. Denote such a vector by $\mathfrak{a}$. Then we write $S_5$ as an appropriate subgroup of $\text{GL}_5(\Z)$ so that our action on $\mathfrak{a}$ is the usual multiplication of a matrix and a vector.  From the point of view of the Coxeter groups, the initial multiset $HD$ is not so important, but we will assume that $a=b=1/2,d=1$ for convenience throughout. Thus, we may write $\mathfrak{a}=(r,s)$ in our setting, with $a,b$ and $d$ implicit, so $F(r,s)=F(\mathfrak{a})$. The Kummer transformation \ref{kummer} in our normalization is $$F(\mathfrak{a})=\alpha_{r,s}\cdot F(K\mathfrak{a})$$ for $K$ the $5\times5$ matrix in the proof of Lemma \ref{cox} and $\alpha_{r,s}$ a gamma quotient given explicitly in (\ref{crs}) below. 

Beyer et al. \cite{coxeter} also study \say{three-term identities}, which originate in the work of Thomae \cite{thomae} and Bailey \cite{bailey}. A typical three-term identity formulated in the language of Coxeter groups by \cite{coxeter} is a theorem of Thomae, $${}_3F_2( \mathfrak{a})(1)=\alpha(
\mathfrak{a})\cdot{}_3F_2(m_1\mathfrak{a})(1)+\alpha'(\mathfrak{a})\cdot {}_3F_2(m_2\mathfrak{a})(1)$$  where $\alpha(\mathfrak{a}),\alpha'(\mathfrak{a})$ are constants depending on the choice of datum $\mathfrak{a}$ and $m_1$ and $m_2$ are fixed $5\times5$  matrices specified in \cite{coxeter}. A more explicit realization of this result is given in Proposition \ref{thomae} below. They also show that for three-term identities, the invariance group is $S_6\times C_2$, where $C_2$ is the group of order 2, but the process of producing identities is more complicated. We discuss a new way to understand certain three-term identities via modular forms extensively in the second half of the paper.

Our three term relations in the context of Beyer et al's work are derived from the Kummer transformation \eqref{kummer} discussed above and the Atkin--Lehner involution. We will only use $W_2$, which is defined by the map $W_2:\tau\mapsto -1/2\tau$. Recall if $r=a/b$ has $(a,b)=1$, then $N=2b$. The action of $W_2$ is \begin{equation}\label{atkinlehner}
    W_2\,\mathbb{K}_2(r,s)(2b\tau)=\tau^{-3}\s \mathbb{K}_2(r,s)(-b/\tau)=2^{8s-16r}(-i)^3\cdot \mathbb{K}_2(s-r,s)(b\tau).
\end{equation} This is straightforward to prove using the well-known formula $\eta(-1/\tau)=\sqrt{-i\tau}\eta(\tau)$. We fix an embedding of $\sqrt{-1}$ into $\C$, which we denote by $i$, for the remainder of this paper. The action of $W_2$ and the Kummer transformation on the pair $(r,s)$
are given by the operators 
\begin{equation}\label{eq:A}
    A: (r,s)\mapsto (s-r,s)
\end{equation} and \begin{equation}\label{eq:K}
    K: (r,s)\mapsto (1-r,1/2-r+s).
\end{equation}   

\begin{proposition}\label{cox}
    The group $G$ generated by $A$ and $K$ is isomorphic to $D_6$, the dihedral group of order 12.
\end{proposition}
\begin{proof}
    Both $A$ and $K$ are involutions, i.e., they have order two. 

      The general version of \eqref{kummer} given in Corollary 3.3.5 of \cite{aar} gives $$K((a,b,c,d,e)^t)=(a,d-b,d-c,d,d+e-b-c)^t. $$ Thus, with respect to the standard basis for $\Q^5$, the matrix for $K$ is   
   $$K=\begin{pmatrix}
    1& 0& 0& 0& 0\\ 0& -1& 0& 1& 0\\ 0& 0& -1& 1& 0\\ 0& 0&
  0& 1& 0\\ 0& -1& -1& 1& 1
\end{pmatrix}.$$ The matrix for $A$ is defined similarly. We also compute that $$AK=\begin{pmatrix}
    1&0&0&0&0\\
    0&-1&0&1&0\\
    0&-1&0&0&1\\
    0&0&0&1&0\\
    0&-1&-1&1&1
\end{pmatrix}.$$ Using linear algebra, it is then easy to check that $A$ and $K$ have order 2 and $AK$ has order 6, and so the group is isomorphic to the dihedral group of order 12.
\end{proof} 

\begin{remark}
    Note $D_6$ is a subgroup of the group $S_6\times C_2$ given in \cite{coxeter} for three term identities. We also point out that $D_6$ as a group is a subgroup of $S_5$ as well; however, using the normalization of \cite{coxeter}, denoted $\hat{F}(r,s)$, we should not view our group $G$ as a subgroup of the invariance group $S_5$. Since $\hat{F}(K(r,s))=\hat{F}(r,s)$ from \cite{coxeter}, $K$ is in the invariance group of $\hat{F}(r,s)$. However, $\hat{F}(A(r,s))\neq \hat{F}(r,s)$ in general. For example, we can see $A(1/8,5/8)=(1/2,5/8)$, and it is easy to check numerically that $\hat{F}(1/8,5/8)\neq\hat{F}(1/2,5/8)$. This implies $A$ is not an element of the invariance group. As a result, besides the matrix $K$ arising from the Kummer transformation, we will use the group $G$ exclusively for three-term identities.
\end{remark} 

According to Beyer et al. Section 2.2 \cite{coxeter}, three-term identities arise from the cosets of a certain normal subgroup of $S_6\times C_2$.  One important subgroup is the commutator $D_6'=[D_6,D_6]$ in $D_6$. We have $$D_6/D_6'\cong \Z/2\Z\times \Z/2\Z=\langle M_1,M_2 \rangle,$$ where \begin{equation}\label{eq:M&N}
     M_1:=AKA, \quad M_2:=KAK.
 \end{equation} Using the matrix form of $A$ and $K$, it is straightforward to verify that these are coset representatives and $M_2,$ $M_1$, and $M_1M_2$ are indeed involutions. Note that the Kummer twist fixes (within the family) the right-hand side of Table \ref{tab}, and the Atkin--Lehner involution fixes the left-hand column of Table \ref{tab}.  This implies $M_1$ and $M_2$ fix the conjugates of the left-hand column of Table \ref{tab}. This construction provides the group theoretical background for Theorem \ref{mainpar}, stated above.

\section{Part I: Modular Forms and Coxeter Groups}\label{modcox}
In this section, we first show Theorem \ref{completepre}. This allows us to construct the Hecke eigenforms and their $L$-values, which are provided in Table \ref{tab:my_label} in the Appendix. Next, we prove Theorem \ref{twistingpre} and provide a classification of all families.

\subsection{Building Eigenforms from the $\mathbb{K}_2$ Functions}\label{building}
\subsubsection{Hecke Operators}
First, we prove that the Hecke operators can be used to construct Hecke eigenforms out of Galois families.  

In Table \ref{tab}, we organize families using an element of the form $(1/b,s'/m)$, which makes sense due to the following Lemma.
\begin{Lemma}\label{1/b}
   For $(i,b)=1$ and $(c_i,m)=1$, any $(i/b,c_i/m)\in \mathbb S_2'$ is conjugate via Definition \ref{conj} to a pair of the form $(1/b,c_1/24)\in \mathbb S_2'$.
\end{Lemma}
\begin{proof}
    The case $b=2$ is trivial, since the only possible $i/2$ lying between $0$ and $3/2$ is 1/2. So assume $b>2$. Note $(1/b)(b-i)+i/b$ equals 1, an integer. Also observe that $b-i$ is coprime to $M$, since $(i,b)=1$. So we only need to prove there is some $c'$ so that $(b-i)s'/m+c_i/m$ is also an integer, that is, $(b-i)c'+c_i\equiv 0\mod m$. For all $(i/b,c_i/m)$, $m\mid 24$, and thus $m$ can be written as products consisting of powers of 2 and 3. Therefore, $(b-i, m)=1$, and it follows that the linear congruence \begin{equation}\label{cong}
        (b-i)c'+c_i\equiv 0\mod m
    \end{equation} has a unique solution modulo $m$. We choose an integer solution $x$ of (\ref{cong}) so that $x\leq m$. If $x/m\geq 1/b$ we only need $x/m<3/2$ so that the conjugate $(1/b,x/m)$ lies in $\mathbb S_2'$. Because we assumed $x\leq m$, we have $x/m\leq 1<3/2$. Therefore, we can take $c'=x$. When $x/m<1/b$, since $b>2$, we have $x/m<1/2$, and it follows that $1/n<x/m+1<3/2$. Moreover, clearly $x+m$ satisfies the congruence (\ref{cong}) as well, so we take $c'=x+m$.
\end{proof}
\begin{remark}\label{four}
    Each $\mathbb{K}_2(a_i/b,s_i)$ has nonzero $n$th Fourier coefficients only if $n\equiv a_i\mod b$. This can be checked by directly expanding the eta product definition of $\mathbb{K}_2(a_i/b,s_i)(N\tau)$.
\end{remark} 
We are now ready to prove the first part of Theorem \ref{completepre}.

\begin{Lemma}\label{hecke2}
    Let ${\mathfrak {S}_{r,s}:=}\{\mathbb{K}_2(i/b,s_i)\}_{i\in (\Z/b\Z)^\times}$ be a Galois family. Let $p$ be a prime so that $p\equiv i\mod b$. Then there exists a nonzero integer $C(p,j)$ so that for any $j\in (\Z/b\Z)^\times$, \begin{equation*}
    T_p\,\mathbb{K}_2(j/b,s_j)=C(p,j) \cdot \mathbb{K}_2(k/b,s_k),
\end{equation*} where $ij=k$ as elements of $(\Z/b\Z)^\times$. In other words,  up to scalar multiplication, the Hecke operators permute the elements in set and the action of the Hecke operators is governed by the group structure of $(\Z/b\Z)^\times$. 
\end{Lemma}
\begin{remark}
    Lemma \ref{hecke2} is a more general version of an example given in Corollary 3.1 of \cite{aglt} for the family $(1/8,1)$. They can use the Dwork dash operator  from the $p$-adic aspect since $s=1$, which has the same action. In this sense, Lemma \ref{hecke2} tells us that we can use the Hecke operators to generalize the Dwork dash operator, which does not apply immediately when $s\neq 1$, to Galois families.
\end{remark} 

\begin{proof}
 The Galois families $\mathfrak{S}_{r,s}$  always satisfy $$T_p\,\mathfrak{S}_{r,s}=\mathfrak{S}_{r,s}\quad\quad p\nmid N.$$ This can be checked on a case by case basis. A more systematic explanation is involves the multiplier system $\psi_{r,s}$ of $\mathbb{K}_2(r,s)$ with respect to $\Gamma_1(N(r))$. Note $N(r)$ is smaller than $\mathcal{N}$ given in Lemma \ref{aglt} and so $\psi_{r,s}$ is not a character, but rather a function on $\Gamma_1(N(r))$ derived from the well-known multiplier system for $\eta(\tau)$. On $\Gamma_1(\mathcal{N})$, $\psi_{r,s}$ becomes the character $\chi$ typically associated to a weight 3 modular forms. The vector space $S_3(N(r),\psi_{r,s})$ of modular forms with multiplier system $\psi_{r,s}$ is a subspace of $S_3(\mathcal{N},\chi)$ which has apparently has dimension equal to $\varphi(M)$ for each of our examples. By Remark \ref{four}, all of the $\mathbb{K}_2(r_i,s_i)$ in the family are linearly independent, and so this implies that $\mathfrak{S}_{r,s}$ generates $S_3(N,\psi_{r,s})$. Because  $S_3(N,\psi_{r,s})$ is fixed by the Hecke operators, so is $\mathfrak{S}_{r,s}$. At any rate, we are assured that $T_p\, \mathbb{K}_2(i/b,s_i)$ can be written as a linear combination of elements in $\mathfrak {S}_{r,s}.$
  
  To prove the explicit action we apply Hecke recursion and use Remark \ref{four}. Denote the $n$th Fourier coefficient of $\mathbb{K}_2(i/b,s_i)$ and $T_p\, \mathbb{K}_2(i/b,s_i)$ by $A_n$ and $B_n$ respectively. Then we have the formula   for $T_p$ on the Fourier coefficients $$B_n=A_{np}+p^2 \chi(p)A_{n/p},$$
   where we assume $A_{n/p}$ is zero if $n/p$ is not an integer and $\chi$ is the character of $\mathbb{K}_2(i/b,s_i)$ on $\Gamma_1(\mathcal{N})$. Then, $A_{np}$ is nonzero only if $n\equiv k\mod b$. To see this, note that for $\mathbb{K}_2(j/b,s_j)$ to be in $\mathbb S_2'$, the only possibilities for $b$ are $b=3,4,8,12,24$ based on Lemma \ref{aglt}. In all of these cases, $(\Z/b\Z)^\times$ is isomorphic to a finite product of $\Z/2\Z$ with itself, and so each element has order 2. Recall $A_n$ is nonzero only if $n\equiv j\mod b$ from Remark \ref{four}. Then for $A_{np}$ to be nonzero, we need $np\equiv j\mod b$ where $p\equiv i$. By assumption $ij=k$ in $(\Z/b\Z)^\times$, and since each element is its own inverse, this implies $j=ik$. Since $p\equiv i\mod b$, we have $np\equiv j\mod b$ if and only if $n\equiv k\mod b$. 
   
We now consider $A_{n/p}$. Analogous to above, since $p\equiv i\mod b$ and $i$ is its own inverse, $1/p$ is equivalent to $i$ in $(\Z/b\Z)^\times$. Thus, for $n/p$ to be congruent to $j\mod b$, we must have $n\equiv ij=k\mod b$ as well. To conclude, by the first paragraph, $T_p\, \mathbb{K}_2(j/b,s_j)$ is a linear combination of $\mathbb{K}_2(h/b,s_h)$ functions in $\mathfrak{S}_{r,s}$, and since $\mathbb{K}_2(k/b,s_k)$ is the only element of its $\mathfrak{G}_{r,s}$ with nonzero coefficients if and only if $n\equiv k\mod b$, we have proved the lemma. 
\end{proof}

As a corollary of Lemma \ref{hecke2}, 
we can formulate the actions of the Hecke operators on the Galois families $\mathfrak S_{r,s}$ as follows. 
\begin{corollary}\label{hecke4}
   Assume $\mathfrak S_{i/b,s_i}=\{\mathbb{K}_2(i/b,s_i)\}_{i\in (\Z/b\Z)^\times}$ is a Galois family. For any primes $p$ and $\ell$ coprime to $N$, 
    $$
     T_\ell T_p= T_pT_\ell =D(p,\ell)T_{[p\ell]},\quad   D(p,\ell) \in \Z,
    $$
    where $[p\ell] \equiv p\ell \mod b$ is defined so that $0<[p\ell]<b$. Therefore, the Hecke algebra on this set is determined by $T_p$, $p\in (\Z/b\Z)^\times$. Moreover,  for a prime $p\in (\Z/b\Z)^\times$, after a suitable reordering of the basis $\mathfrak S_{i/b,s_i}$, the action of $T_p$ is given by
    $$
       \begin{pmatrix}
             T_{p,1} &0& \dots & 0\\
             0& T_{p,2} & \ddots & 0\\
               0&\ddots & \ddots & 0\\
             0&\dots& 0&  T_{p,\frac{\phi(b)}2}
       \end{pmatrix},  \quad  
       T_{p,i}=
       \begin{pmatrix}
             0 &a_i\\
             b_i& 0\\
       \end{pmatrix}, \quad  a_ib_i = D(p,p). 
    $$   

Also, when $p\equiv 1\mod b$ is a prime, the action of $T_p$ on $\mathfrak S_{i/b,s_i}$ is $C(p,1)\mathit{Id}$, where Id is the identity matrix. 
\end{corollary}

\begin{proof}
 For this proof, write $K_i=\mathbb{K}_2(i/b,s_i)$. Note $$T_pT_\ell \, K_j=C(\ell,j)C(p,[j\ell])K_{[p \ell j]}$$ by standard properties of the Hecke operators and Lemma \ref{hecke2}. On the other hand, $$T_{[p\ell]}K_j=C([p\ell],j)K_{[[p\ell]j]},$$ since every integer less than $b$ and coprime to $b$ happens to be a prime number in all of our cases. But $[[p\ell]j]=[p\ell j]$ in $(\Z/b\Z)^\times$, and so we have $$T_{\ell}T_p\, K_j=\frac{C(\ell,j)C(p,[j\ell])}{C([p\ell],j)}T_{[p\ell]}K_j.$$ This constant does not depend on $j$ because the Hecke algebra is commutative, and it is by definition $D(p,\ell)$ as stated in the Corollary. The remainder of the corollary follows immediately from Lemma \ref{hecke2}.
We now prove the remainder of Theorem \ref{completepre}
\end{proof}

\begin{theorem}\label{complete}
      Assume $\{\mathbb{K}_2(i/b,s_i)\}_{i\in (\Z/b\Z)^\times}$ is a Galois family and $r=1/b$, $s=s_1$. Then there exist algebraic numbers $\beta_i$ integral over a quadratic number field depending on $i$ so that $$f_{r,s}=\sum_{i\in (\Z/b\Z)^\times}\beta_i \mathbb{K}_2(i/b,s_i)$$ is an eigenvector of $T_p$ for all $p$ coprime to 6. Moreover, any Hecke eigenform constructed as above can be written as $f_{r,s}\otimes \phi$, where $\phi$ is a quadratic character with conductor dividing $b$. 
\end{theorem}
\begin{proof}
By Corollary \ref{hecke4}, the eigenvalues of $T_p$ can only be $\pm \sqrt{D(p,p)}$ for each $p$. 
For all $i\in (\Z/b\Z)^\times$, the only possible combination of $\mathbb{K}_2(1/b,s_1)$ and $\mathbb{K}_2(i/b,s_i)$ that is an eigenvector of $T_i$ with eigenvalue $\pm\sqrt{D(i,i)}$ is $$\mathbb{K}_2(1/b,s_1)\pm \sqrt{D(i,i)}\mathbb{K}_2(i/b,s_i)$$ due to Lemma \ref{hecke2}. Because the matrices for $T_i$ are simultaneously diagnolizable, by the spectral theorem, there is a basis for $\mathfrak S_{1/b,s_1}$  which are eigenvectors for all the $T_i$, and so we can construct precisely $\varphi(b)$ linearly independent Hecke eigenforms from the family. We describe these eigenforms explicitly below. For now, we note by the discussion above, each such eigenvector must have the form $$\sum_{i\in (\Z/b\Z)^\times}\pm \sqrt{D(i,i)} \mathbb{K}_2(i/b,s_i).$$ The coefficients $\beta_i$ in the statement are thus $\sqrt{D(i,i)}$ up to a sign. Following Corollary \ref{hecke4}, we arbitrarily choose signs for the $\beta_i$ so that if $ij=k$ in $\Z$, then $\beta_i\beta_j=\beta_k$, and fix this as the Hecke eigenform $f_{r,s}$. Our choice of sign for $\beta_i$ is specified in Table \ref{tab:my_label}.

Since any element of $(\Z/b\Z)^\times$ has order 2 as noted previously, every character on the group is quadratic. Moreover $(\Z/b\Z)^\times$ is isomorphic to its character group, and so there are precisely $\varphi(b)$ distinct quadratic characters of conductor $b$, $\phi_1,...,\phi_{\varphi(b)}$, in bijection with the elements of $(\Z/b\Z)^\times$. The tensor product $f_{r,s}\otimes \phi_j$ for any $j$ swaps the signs of the $\beta_i$, and so is also a Hecke eigenform constructed from $\mathfrak{S}_{r,s}$. These are $\Q$-linearly independent, and since the dimension of $\mathfrak{S}_{r,s}$ is also $\varphi(b)$, these provide a basis for $\mathfrak{S}_{r,s}$ and hence describe every Hecke eigenform that can be constructed from the family.
\end{proof}

Because of (\ref{lval}), integrating both sides produces an exact $L$-value of the Hecke eigenform.

\begin{corollary}\label{lval2}
    Assume $\{\mathbb{K}_2(i/b,s_i)\}_{i\in (\Z/b\Z)^\times}$ is a Galois family and $f_{r,s}$ is the Hecke eigenform constructed in Theorem \ref{complete}. Then for any quadratic character $\phi$ with conductor dividing $b$, $$L(f_{r,s}\otimes \phi,1)=\sum_{i\in (\Z/b\Z)^\times}\phi(i)\cdot \beta_i\cdot  F(i/b,s_i)$$ where $ F(i/b,s_i)$ is defined in \eqref{frs}. 
\end{corollary}
Using the functional equation for a weight 3 Hecke eigenform, we can obtain the special $L$-value at 2 from this formula as well.

\smallskip

\subsubsection{An Example of Constructing Hecke Eigenforms} 

To illustrate the Theorem above, we use the  family that includes the pair $(1/8,5/8)$, in class 6 of Table \ref{tab} below. The family is Galois, as the conjugates have the form $(a/8,a/8+1/2)$ for $a=1,3,5,7$. There are always $\varphi(M)$ conjugates, and so it suffices to find these four as $\varphi(8)=4$.

 To find the linear combination of the conjugates,  it suffices to determine the action of the Hecke operators $T_p\, \mathbb{K}_2(1/8,5/8)$ for $p=3,5,7$. A computation on the Fourier coefficients reveals the following action. As before, let $K_i=\mathbb{K}_2(i/b,s_i)$.
$$\begin{array}{c||c|c|c|c}
&K_1&K_3&K_5&K_7\\\hline
T_3&12K_3&K_1&-4K_7&-3K_5\\
T_5&-48K_5&16K_7&K_1&-3K_3\\
T_7&-64K_7&16K_5&-4K_3&K_1
\end{array}
$$
From the table, we see that $T_3^2=12$, $T_5^2=-48$, and $T_3T_5=-3T_7$. We conclude $\beta_1=2\sqrt{3}$, $\beta_2=4i\sqrt{3}$ and because of the multiplicative nature of the coefficients,  $\beta_3=-\beta_1\beta_2/3=-8i$. We find that
$$f_{1/8,5/8}=\mathbb{K}_2(1/8,5/8)+2\sqrt{3}\mathbb{K}_2(3/8,7/8)+4\sqrt{3}i\mathbb{K}_2(5/8,9/8)-8i\mathbb{K}_2(7/8,11/8).$$ Since the Hecke eigenvalue field $K_f$ of $f_{1/8,5/8}$ visibly has degree 4 from above, all quadratic twists with conductor dividing $b$ of $f_{1/8,5/8}$ are \textit{inner twists}, that is, they correspond to automorphisms of $K_f$. Hecke eigenforms differing by an inner twist are linearly dependent over $K_f$, and so all inner twists of a modular form are given by a single LMFDB label. Using the level and character provided in Lemma \ref{aglt}, we determine that $f_{1/8,5/8}$ necessarily is associated to the LMFDB label $64.3.d.a$. However, when $K_f$ has degree less than $\varphi(b)$, sometimes multiple LMFDB labels will appear for a single family $\mathfrak{S}_{r,s}$. These are also listed in Table \ref{tab} when applicable.

By Corollary \ref{lval2}, the above formula also immediately transforms into the special $L$-value at 1 if we replace $\mathbb{K}_2$ with $F$. That is, $$16L(f_{1/8,5/8},1)=F(1/8,5/8)-2\sqrt{3}F(3/8,7/8)- 4\sqrt{3}iF(5/8,9/8)-8iF(7/8,11/8).$$

\begin{remark}\label{lmfd}
    One must be very careful when referring to the $L$-value of an LMFDB label, since this could potentially refer to any of four distinct $L$-values corresponding to the inner twists of the Hecke eigenform, or even at times the product of all inner twists. Although we are often able to obtain the $L$-values for inner twists using our method, it is important to note that results like Theorem \ref{relations} depend on the inner twist chosen, and are not true as stated for a different inner twist. To avoid this ambiguity, we use our notation $f_{1/8,5/8}$ which fixes a twist in Table \ref{tab:my_label}, instead of the notation $f_{64.3.d.a}$ as in \cite{aglt} and \cite{grove} to refer to the Hecke eigenform we constructed above. We provide the LMFDB label(s) for each family only in Table \ref{tab}.
\end{remark}

Using Lemma \ref{l1} below, this $L$ -value can be simplified to $$L(f_{1/8,5/8},1)=\frac{(1-\sqrt{-2}-\sqrt{3})}{16}[2F(3/8,7/8)+4iF(5/8,9/8)].$$ For convenience, the correct linear combination of conjugates for all the Galois cases is listed in the Appendix.

\subsection{Twisting and the Kummer Transformation}
Note that $K$ interchanges the families $\{(i/4,1)\}_{i=1,3}$ and $\{(i/4,i/4+1/2)\}_{i=1,3}$. As observed in the introduction that we have the following twisting property between these families of conjugates.  \begin{align*}
     \mathbb{K}_2(1/4,1)=q - 2 q^5 - 7 q^9 + 14 q^{13} + 18 q^{17} - 32 q^{21} - 21 q^{25} + 14 q^{29}+O(q^{30})\\
    \mathbb{K}_2 (1/4,3/4)=q + 2 q^5 - 7 q^9 - 14 q^{13} + 18 q^{17} + 32 q^{21} - 21 q^{25} - 14 q^{29}+O(q^{30}).
 \end{align*} 
 In other words, when $K$ interchanges two families, the modular forms in those families differ by a sign. The relationship is summarized in Figure \ref{fig:placeholder}. 
 \begin{figure}
   \centering
    \includegraphics[scale=.4]{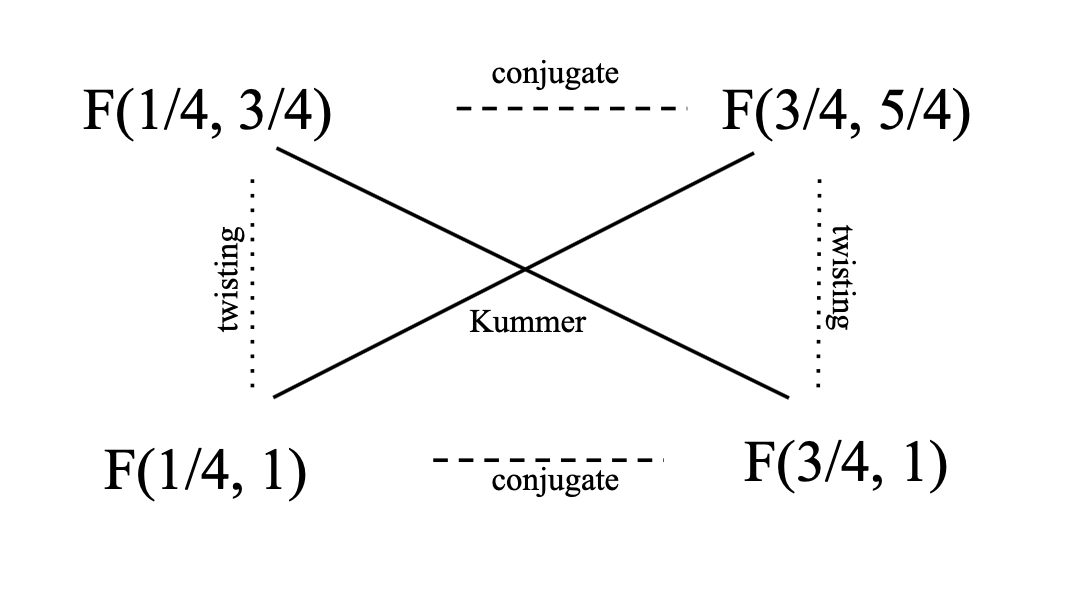}
    \caption{}
     \label{fig:placeholder}
 \end{figure}

In the diagram, the dotted lines denote relations between modular forms (Theorem \ref{twisting} below), the dashed lines are conjugates, and the solid lines are the Kummer transformation. Our application to $L$-values of newforms for some special cases arises because of a combination of all three of these aspects. In this section, we focus on the twisting. This is explained by symmetries of the hypergeometric functions underlying $\mathbb{K}_2(r,s)$.  

\begin{theorem}\label{twisting}
   Let $h=h(r,s)=r-s+3/2$. Then if $0<h(r,s)<3/2$, for any positive integer $n$, the $n$th Fourier coefficients of $\mathbb{K}_2(r,s)(\tau)$ and $\mathbb{K}_2(r,h)(\tau)$ are the same up to a sign. Moreover, if $a$ is the numerator of $r$, the $n$th coefficients are equal if and only if $n\equiv a\mod N$, where $N$ is as in (\ref{N}).
\end{theorem} We will say that $(r,s)$ and $(r,h)$ are \textit{Kummer twists} of each other in this case, which are labeled \say{twisting} in the figure. Note this is a slightly more specific version of Theorem \ref{twistingpre} from the introduction.

\begin{proof}
    Apply the Pfaff transformation (see \cite{aar}), $$\pfq{2}{1}{1/2,1/2}{,1}{x}=(1-x)^{-1/2}\pfq{2}{1}{1/2,1/2}{,1}{\frac{x}{x-1}},$$ to the right hand side of $$\mathbb{K}_2(r,s)(\tau)d\tau=\lambda^r(1-\lambda)^{s-r-1}\pfq{2}{1}{1/2,1/2}{,1}{\lambda}\frac{d\lambda}{2\pi i\lambda}.$$ This gives us $$\lambda^{r}(1-\lambda)^{s-r-3/2}\pfq{2}{1}{1/2,1/2}{,1}{\frac{\lambda}{\lambda-1}}\frac{d\lambda}{2\pi i\lambda}.$$ Making a change of variable, $t=\lambda/(\lambda-1)$, and using that $\lambda(\tau+1)=\lambda/(\lambda-1)$ (see e.g. \cite{yoshida}), we compute \begin{align*}
        \mathbb{K}_2(r,s)(\tau)d\tau=&\left(\frac{t}{t-1}\right)^r\left(\frac{1}{t-1}\right)^{s-r-3/2}\pfq{2}{1}{1/2&1/2}{&1}{t}\frac{dt}{2\pi it(t-1)}\\&=(-1)^r(1-t)^{1/2-s}t^r\pfq{2}{1}{1/2&1/2}{&1}{t}\frac{dt}{2\pi it}\\&=(-1)^r\mathbb{K}_2(r,r-s+3/2)(\tau+1)d\tau,
    \end{align*}
    invoking that the linear fractional transformation $\lambda/(\lambda-1)$ is its own inverse repeatedly.  
    Because of this, the $q$-expansions of $\mathbb{K}_2(r,s)(\tau)$ and $\mathbb{K}_2(r,h)(\tau)$ differ by a root of unity, and as the $\mathbb{K}_2(r,s)$ functions always have rational coefficients, this means the Fourier coefficients must differ by a sign.

     Let $r=a/b$ for $a$ and $b$ coprime integers. For the second part, the key observation is that $b$ is always equal to $N/2$, as noted after (\ref{N}). Also, note that $(-1)^r=e^{-a\pi i/b}=e^{-2a\pi i/N}.$ We temporarily---only for the remainder of this proof---use the local uniformizer $q_N=e^{2\pi i\tau/N}$.
    The map $\tau\mapsto\tau+1$ sends $$e^{2\pi i\tau/N}\mapsto e^{2\pi i\tau/N}e^{2\pi i/N},$$ and so $q_N^n\mapsto e^{2n\pi i/N}q_N^n$. Multiplying $e^{-2a\pi i/N}$ through, we see the $n$th Fourier coefficient for $\mathbb{K}_2(r,s)(\tau)$ is the same as the $n$th Fourier coefficient of $\mathbb{K}_2(r,h)(\tau)$ up to $e^{2(n-a)\pi i/N}$. Again writing out the $q_N$-expansion, we see the first term in the Fourier expansion of $\mathbb{K}_2(r,s)(\tau)$ is $q_N^a$, and all other non-zero coefficients are equivalent to $a$ modulo $b$. Since $b=N/2$, this implies $n=a+mN/2$ for some integer $m$. This divides the $n$ into two congruence classes mod $N$, equivalent to $a$ when $m$ is even and $a+N/2$ when $m$ is odd. Evidently, the signs do not switch if and only if $m$ is even, so this proves the claim.
\end{proof}

Note that this doesn't even require our modular form to be in $\mathbb S_2'$, and
the twist could be the same as the original modular form. 
However, the phenomenon is called twisting because in all but one of the Galois cases, the completed (as in Theorem \ref{complete}) two $q$-series differ by a twist of a character. For example, the newform \begin{equation}\label{1/4}
f_{1/4,1}=\mathbb{K}_2(1/4,1)+4i\mathbb{K}_2(3/4,1)
\end{equation} is related by Kummer to $$f_{1/4,3/4}=\mathbb{K}_2(1/4,3/4)+4i \mathbb{K}_2(3/4,5/4)$$ and $f_{1/4,1}=\phi \otimes f_{1/4,3/4}$, where $\phi$ is the quadratic character modulo 4. The one exception is class 3 in Table \ref{tab}, which occurs because one of the corresponding Hecke eigenforms is an oldform for the level specified in Lemma \ref{aglt}.

We rely on the following idea for our classification.
\begin{definition}
    We say  a family of conjugates $(r_i,s_i)$ is \textit{self-twisted} if the Kummer twist of $(r_i,s_i)$ equals $(r_j,s_j)$ for some $j$ in the same family.
\end{definition}

 All the non-CM self-twisted families arise from applying the Atkin--Lehner involution $W_2$ to a Galois family. For example, the Atkin--Lehner involution maps $\mathbb{K}_2(a/8,a/8+1/2)$ to $\mathbb{K}_2(1/2,a/8+1/2)$, up to a constant. We can then easily check that the latter family is self-twisted. Each of these families has the same number of conjugates, so there should be a relation between them. See for instance Lemma \ref{threeterm}. This allows us to compute the action of the Hecke operators indirectly.

\subsection{Classification of Conjugate Families}\label{class} In this section, we complete the classification of our families by enumerating all the possible families. Understanding our situation in terms of the group action of $D_6:=\langle A,K\rangle$ on $\mathbb{S}_2'$ \eqref{s2} allows us to classify our modular forms in a natural way. We demonstrate this pictorially by constructing a graph where the nodes are the elements of $\mathbb{S}_2'$. The edges are the pairs related by $A$ (colored red) or by $K$ (colored black), and a black dotted line if two vertices are conjugate.  An example is given in Figure \ref{fig:galois} below. 
 For example, $(1/4,5/4)$ corresponds to a level 16 holomorphic modular form. However, applying $K$ to this produces $(3/4,3/2)$, which is not in $\mathbb{S}_2'$. The figure shows first a specific connected component of the graph, and then a more general connected component.

\begin{figure}[htbp]
    \centering
    \includegraphics[width=0.8\linewidth]{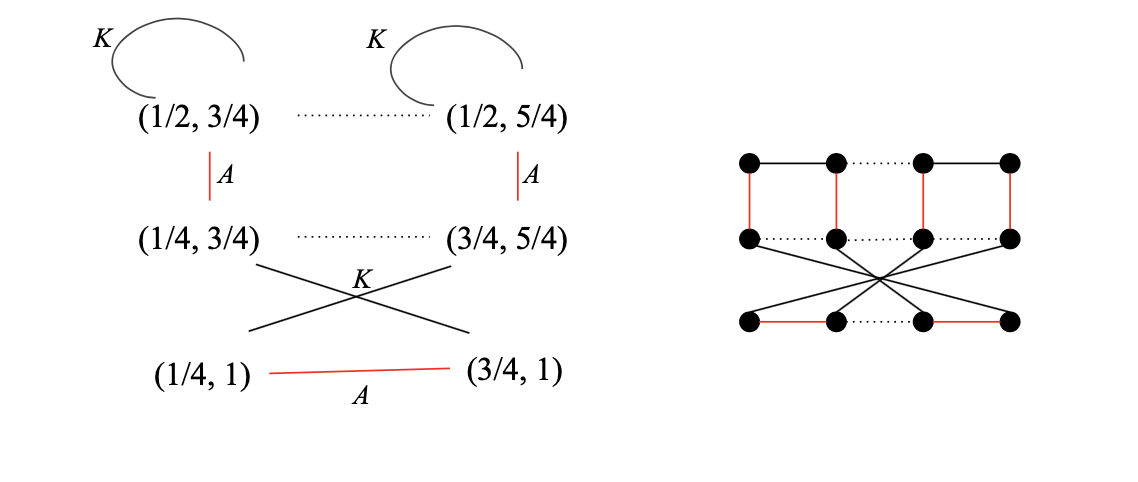}
    \caption{The graph associated to class 5 on the left, with the Atkin--Lehner involution in red, the Kummer twist in black, and elements in the same Hecke orbit dotted. The graph on the right is associated to a generic non-CM four-term class} 
    \label{fig:galois}
\end{figure}

From here on, when we refer to a \say{family} of pairs, we mean the set of conjugates for a given pair, and when we say a \say{class} of pairs, we mean all pairs in a connected component of the graph, i.e. a row in Table \ref{tab}. When we refer to a Galois class, we mean any class containing at least one Galois family. We only label the classes in Table \ref{tab}. Individual families are assigned labels in Table \ref{tab:my_label} of the Appendix.

\begin{definition}
  For a representative pair $(r,s)$ in a  given family of conjugates, we denote its stabilizer subgroup in $D_6$ by $\text{Stab}(r,s)$.
\end{definition}
Note this is stabilized within the family. For example, we would consider $A$ to be in $\text{Stab}(1/8,1)$ because $A(1/8,1)=(7/8,1)$, which is conjugate to $(1/8,1)$. We can classify the stabilizer as follows.
\begin{proposition}
    For every pair  $(r,s)$ in Table \ref{tab} outside of Class 1, we have 
    $$\text{Stab}(r,s)\cong \Z/2\Z\times \Z/2\Z. $$ If $(r,s)$ is in Class 1, then  $$\text{Stab}(r,s)\cong D_6. $$
\end{proposition}
Thus, in the Galois classes excluding Class 1, $$|D_6/\text{Stab}(r,s)|=3.$$  

Here is the main result. 
\begin{theorem}\label{classs}
There are 18 Galois families falling into 10 connected components. Among these, 12 are non-CM,  falling into 6 connected components, 3 are CM and non-degenerate, falling into 2 connected components, and 4 are CM and degenerate, falling into 2 connected components. 
\end{theorem}
We defer the proof of Theorem \ref{classs} until after we have stated the tables of families obtained using our method.

 \subsubsection{List of all Galois Classes}
We count the Galois families below.

Table \ref{tab} below records the list of all Galois families using the classification above.  The three columns correspond to the size of $D_6/\text{Stab}(r,s)$. Thus, each row of the table corresponds to a connected component of the graph described above, including all three types of edges.  We also include the families generated by applying the Atkin--Lehner involution to a Galois family in this table, since these are in the same connected component of the graph as the Galois family.

\begin{table}[ht]
   \begin{center}
     \begin{tabular}{c|cc|cc|c|c}
       Label & Data & LMFDB & K-Data & K-LMFDB&AL-K-Data& CM? \\\hline
        1 & $1/2,1$ & 16.3.c.a & $1/2,1\s (*)$ & 16.3.c.a &$1/2,1\s (*)$& yes\\\hline
        2& $1/4,5/4$ & 16.3.c.a & $1/4,1/2$ & 64.3.c.a & $1,5/4$ & yes\\\hline
         3&  $1/3,7/6$ & $\begin{array}{c} \text{36.3.d.a}\\\text{36.3.d.b} \end{array}$ &$1/3,2/3$ & $\begin{array}{c} \text{36.3.d.a}\\\text{36.3.d.b} \end{array}$ & $1/6, 5/6$ & yes\\\hline
          4& $1/8,9/8$ & 32.3.d.a & $1/8,1/2$ & 256.3.c.e & $1,9/8$ & yes\\\hline
       5 & $1/4,3/4$ & 32.3.c.a & $1/4,1$ & 64.3.c.b& $1/2,3/4$ & no \\\hline
       6& $1/8,5/8$ & 64.3.d.a & $1/8,1$ & 256.3.c.g & $1/2,5/8$ & no \\\hline
       7& $1/8,7/8$ & 128.3.d.c & $1/8,3/4$ & $\begin{array}{c} \text{256.3.c.c}\\\text{256.3.c.f} \end{array}$& $1/4,5/8$ & no\\\hline

         8& $1/12,11/12$ & $\begin{array}{c} \text{288.3.g.a}\\\text{288.3.g.c} \end{array}$ & $1/12,2/3$ & $\begin{array}{c} \text{576.3.g.d}\\\text{576.3.g.h} \end{array}$& $1/6,7/12$ & no\\\hline
          9& $1/24, 17/24$ & 288.3.b.c & $1/24, 5/6$ & 2304.3.g.y& $2/3,17/24$ & no\\\hline
       10& $1/24,23/24$ & 1152.3.b.i & $1/24,7/12$ & $\begin{array}{c} \text{2304.3.g.p}\\\text{2304.3.g.w} \end{array}$ & $1/12,13/24$ &  no\\

    \end{tabular}
\end{center}
    \caption{Classes of Galois Families. The three columns for these classes denote the action of the Kummer twist and the Atkin--Lehner involution, respectively, on the initial data. By Theorem \ref{complete}, we can construct any twist by a quadratic character of our initial modular form if the denominator of $r$ divides the conductor.  When this is an inner twist, the LMFDB label is the same, but when it is not an inner twist, the label is different. For completeness, we list both labels in these cases. }
    \label{tab}
\end{table}

 To prove that Table \ref{tab} encompasses all possible Galois families, one can simply enumerate every other family and show they are not Galois or given by applying the Atkin-Lehner involution to a Galois family.
 
\subsubsection{List of non-Galois classes}
 The notion of conjugates generates the following families. We can also obtain another four families by applying the Atkin-Lehner involution to each of these families. 
\begin{table}[ht]
    \centering
    \begin{tabular}{ccccccccc}
        & Data  & LMFDB 1 & LMFDB 2 & K-Data & K-LMFDB 1 & K-LMFDB 2\\\hline
        11 & $1/6,4/3$ & 144.3.g.c & - & $1/6,1/3$ & 144.3.g.c & -\\\hline
        12 & $1/24,5/24$ & 576.3.b.b & 576.3.b.c & $1/24,1/3$ &2304.3.g.t &2304.3.g.s\\\hline
        13 & $1/12,17/12$ &36.3.d.b  & 144.3 g.c & $1/12,1/6$ & 576.3.g.c & 576.3.g.f \\\hline
        14 & $1/8,11/8$ & 128.3.d.a & 128.3.d.b & $1/8,1/4$ & 256.3.c.b & 256.3.c.d\\\hline
        15 & $1/24,35/24$ & 1152.3.b.b & 1152.3.b.f & $1/24,1/12$ &2304.3.g.e & 2304.3.g.i\\
    \end{tabular}
    \caption{Non-Galois classes. Families arising from the Atkin--Lehner involution are not shown, but can be found using \eqref{atkinlehner}. Note that both modular forms listed have CM. More details about these families are given in \cite{rosen2}.}
    \label{nongalois}
\end{table}

Referring to the Tables listed above, we can finish the classification.

\begin{proof}[Proof of Theorem \ref{classs}]
         Each nondegenerate class excluding class 1 in Table \ref{tab} has a total of $3\varphi(M)$ pairs, with $\varphi(M)$ in each family. The exceptions are class 1, which only has 1 pair, and the degenerate classes: class 2 has 3 pairs, and class 4 has 6 pairs.  So, in total 106 pairs $(r,s)$ occur in Table \ref{tab}. Counting the number of pairs arising from conjugates to the pairs from non-Galois classes besides class 11 and the Atkin-Lehner involutions of these conjugates, we obtain an additional $3\varphi(M)$ pairs automatically, giving us 72 more pairs, as well as another 3 pairs for class 11. In addition, there are pairs of the form $(r+1,s+1)$ for a pair $(r,s)$ belonging to a non-Galois family---see \cite{rosen2} for some examples. In these cases, the modular form $\mathbb{K}_2(r+1,s+1)$ is a linear combination of $\mathbb{K}_2(r,s+1)$, and $\mathbb{K}_2(r,s)$, and so these pairs do not count towards the $\varphi(M)$ linearly independent pairs guaranteed in each family, but still do count as elements of $\mathbb S_2$. For each non-Galois family excluding class 11, we count a total of $\varphi(M)/4$ such pairs, and so a total of 18 occur. Adding this together, we see that our tables account for $106+72+3+18=199$ pairs, which is the size of $\mathbb{S}_2$'. 
\end{proof}

\section{Part II: Identities of \textit{L}-values }

In this section, we give the proof of the extension of Kummer's theorem to $L$-values, stated precisely as Theorem \ref{relations} below.

We can explicitly find $L$-value relations arising from Theorem \ref{twisting} in the non-CM setting by using the Kummer Transformation \eqref{kummer}, which states  
$$F(r,s)=2^{4-8r}\alpha_{r,s}F(1-r,1/2+q-r)$$  in terms of $F(r,s)$, where \begin{equation}\label{crs}
    \alpha_{r,s}=\frac{\Gamma(r) \Gamma(s - r)^2}{\Gamma(1 - r) \Gamma(s-1/2)^2 }. 
\end{equation} In most cases, $\alpha_{r,s}$ does not appear to be algebraic. Among the congruence examples, the only times it is algebraic, excluding degenerate cases, can be summarized by the fact $$F(r,1)=2^{-4+8r}\csc(\pi r)F(1-r,3/2-r)$$ and $$F(r,1-r)=2^{-4 r} \csc(\pi r)F(1-r,3/2-2r),$$ for any rational $r$.
These correspond to classes 5, 6, 7, 8, and 10 in Table \ref{tab}.

The simplest non-CM example is class 5, which is already addressed in \cite{aglt}. Using Corollary \ref{lval2}, we have $$L(f_{1/4,3/4},1)=F(1/4,3/4)+4iF(3/4,5/4)$$ and $$L(f_{1/4,1},1)=F(1/4,1)-4iF(3/4,1).$$ In this case, \eqref{kummer} takes the form 
\begin{equation}\label{eq:16}
\begin{split}
    F(1/4,3/4)&=4\sin(\pi/4)F(3/4,1)\\ F(3/4,5/4)&=\frac{1}{4}\sin(3\pi/4)F(1/4,1).   
\end{split}
\end{equation}  The critical fact is then that $\sin(\pi/4)=\sin(3\pi/4)$ and so we can factor out $i\sin(\pi/4)$,
proving that \begin{equation}\label{eq:17}
    L(f_{1/4,3/4},1)=-\sin(\pi/4)[-4iF(3/4,1)+F(1/4,1)]=\frac{\sqrt{2}i}{2}L( f_{1/4,1},1).
\end{equation}  In other words, we \say{extend} the identities  \eqref{eq:16} to an identity of $L$-functions, \eqref{eq:17}, $$L(f_{1/4,3/4},1)=i\sin(\pi/4)\cdot L(f_{1/4,1},1).$$ These relations are predicted by Theorem \ref{twistingpre} combined with the following theorem, which we have specialized to our case. \begin{theorem}[Shimura \cite{shimurazeta}]\label{shimura}
 Suppose $f$ is a newform of weight $k>2$ and $\chi$ a finite order character so that $\chi(-1)=1$,  which is referred to as being even.  Then $$\frac{L(f,1)}{L(f\otimes \chi,1)}=\alpha\cdot  g(\chi)^{-1}$$ where $g(\chi)$ is a Gauss sum and $\alpha=\alpha(f,\chi)$ is an algebraic number depending on $f,\chi$. Moreover, if $K_f$  is the field generated by the Hecke eigenvalues of $f$, then $\alpha$ lies in $K_f$, and if $\sigma$ is a nontrivial automorphism of $K_f$, then $$\frac{L(f^\sigma,1)}{ L(f^\sigma\otimes \chi^\sigma,1)}=\alpha^\sigma\cdot  g(\chi^\sigma)^{-1}.$$
\end{theorem}
The constant $\alpha$ is not explicit in Shimura's theorem.

When the denominator of $r$ is e.g. 8 (classes 6 and 7), $\sin(\pi/8)$ does not equal $\sin(3\pi/8)$, and so we cannot use the Kummer transformation directly. To generalize this to classes 7, 8, and 9, we require the three term identities in Theorem \ref{mainpar}.  After doing this, the Kummer transformation extends to an identity of $L$-functions as predicted by Shimura. This occurs because of the deep motivic meaning of the modularity result in \cite{aglt} relating the two hypergeometric series to the two newforms,
translated from the \'etale setting into the complex (de Rham) setting. 

The first section below proves Theorem \ref{mainpar} and the other necessary three-term relations. Then, we prove Theorem \ref{relations}.

\subsection{The Key Three-Term relations}
We write the three-term identities in terms of four series for conciseness, but by taking the real and imaginary parts of each side produces two three-term identities for each identity listed below.

Here is the first relation for class 6: 
\begin{Lemma}\label{l1}
    $$F(1/8,5/8)-8iF(7/8,11/8)=(1-\sqrt{-2})(2F(3/8,7/8)+4iF(5/8,9/8)).$$
\end{Lemma}

\begin{proof}
The newform $f_{1/8,5/8}$ has one even inner twist $\phi_{even}$, given by the quadratic Dirichlet character 8.b in the LMFDB. For this proof, take $f=f_{1/8,5/8}$. Then, since the twist is even, we have by Shimura's result that $L(f\otimes \phi_{even},1)$ and $L(f,1)$ differ by an algebraic number. The Gauss sum of the Dirichlet character $\phi_{even}$ is $\sqrt{8}$, and so Shimura further predicts that $\sqrt{8}\alpha$ belongs to the Hecke eigenvalue field of $f$, which is $\Q(\zeta_{12})$. We may write $L(f\otimes \phi_{even},1)$ and $L(f,1)$ in terms of hypergeometric series using Corollary \ref{lval2}, as $\phi_{even}$ is a quadratic character with conductor dividing 8, which we cross-check with the conductor of 8.b in the LMFDB. Then, rearranging the relation $L(f,1)=\alpha L(f\otimes \phi_{even},1)$, we deduce $$ F(1/8,5/8)-8iF(7/8,11/8)=\frac{\sqrt{3}(1-\alpha)}{1+\alpha}(2F(3/8,7/8)+4iF(5/8,9/8)).$$ Now, let $\sigma$ denote the automorphism of $\C$ obtained by lifting the unique nontrivial automorphism of $\Q(\sqrt{3})$. By construction, $f^\sigma=f\otimes \phi_{even}$. The second part of Shimura's theorem implies $$(\sqrt{8}L(f,1)/L(f\otimes \phi_{even},1))^\sigma=(\sqrt{8}\alpha)^\sigma=\sqrt{8} L(f^\sigma,1)/L(f,1)=8/(\sqrt{8}\alpha).$$ 

Let $\alpha':=\sqrt{8}\alpha$,  which is an element of $\Q(\zeta_{12})$ by Shimura. In general, we may write elements of $\Q(\zeta_{12})$ as $\alpha'=a+b\sqrt{3}$ for $a,b\in \Q(i)$, and then $$1/\alpha'=\frac{a-b\sqrt{3}}{a^2-3b^2}.$$ However, from above $\sigma(\alpha')=a-b\sqrt{3}=8/\alpha'$, and so $a^2-3b^2=8$. This implies the norm $N_{\Q(\zeta_{12})/\Q(i)}(\alpha')=8$, and so the absolute norm is 64. Up to multiplication by a unit in the ring of integers, the only element of norm 64 in $\Q(\zeta_{12})$ is $2+2 i$. This is because 2 is a divisor of 64, so 2 is ramified in $\Q(\zeta_{12})$ and factors as $(2)=(1+i)^2$. Moreover, the unit group of $\Q(\zeta_{12})$ has rank 1 and is generated by $\zeta_{12}-1$. Combining these two facts, we get that $[\alpha'/(2+2i)]^{12}=(\zeta_{12}-1)^n$ for some integer $n$. The elements of the unit group lie in a lattice in Minkowski space, so we should be able to determine them numerically. In particular, let $B=|\zeta_{12}-1|$, where the bars denote complex modulus. Then we have $$\log_B|[\alpha'/(2+2i)]^{12}|=n.$$ Since $n$ is an integer and the integers are discrete in the reals, there is an open ball of radius, say, 1/4, around $n$ so that $n$ is the only integer in this interval. We can compute $\log_B|[\alpha'/(2+2i)]^{12}|$ numerically to a high enough level of precision that we can determine it is within the ball of radius 1/4 around 12, and so we conclude $n=12$. Therefore, we have $\alpha'/(2+2i)=(\zeta_{12}-1)\epsilon$, where $\epsilon$ is a 12th root of unity. The 12th roots of unity can be distinguished numerically, and so we check that $\epsilon=-\zeta_{12}^{10}$. In conclusion, $\alpha'=-(2+2i)(\zeta_{12}-1)\zeta_{12}^{10}=-2 i (-1 + \sqrt{3})$. From this, we can unravel the desired identity.
\end{proof}

There is a similar identity for the Kummer twist of the family above,

\begin{Lemma}\label{l2}
     $F(1/8,1)+8iF(7/8,1)=-2(1 - \sqrt{-2}) (1 + \sqrt{2})(F(3/8, 1) - 2 i F(5/8, 1)).$
\end{Lemma}

\begin{proof}
     Apply the Kummer transformation to each side of the above Lemma, which in this case has the form $F(1-r,3/2-r)=\sin(\pi r)2^{4-8r}F(r,1),$ so we get $$\sin(\pi/8)(8F(7/8,1)+iF(1/8,1))=2\cos(\pi/8)(1-\sqrt{-2})(2F(5/8,1)+iF(3/8,1)).$$ Multiply though by $i$ and note that $\cot(\pi/8)=1+\sqrt{2}$ and we arrive at the desired identity.
\end{proof}

The proof in class 7 is very similar to class 6.

\begin{Lemma}\label{l3}
 $$F(1/8, 7/8) + 8i\sqrt{2} F(7/8, 9/8) = 
 \zeta_8 (2\sqrt{2} F(3/8, 5/8) + 4i F(5/8, 11/8))$$
\end{Lemma}
\begin{proof}
    Similar to before, we twist $f_{1/8,7/8}$ by the even character $\phi_{even}$, which is identical to the character with LMFDB label 8.b used in Lemma \ref{l1}. Shimura's theorem gives us an algebraic relation between the $L$-values.  We again fix $f$ to be $f_{1/8,7/8}$ for this proof. We still have $g(\phi_{even})=\sqrt{8}$, and again let $\alpha$ be the constant relating the two $L$-values, and $\alpha'=\sqrt{8}\alpha$. Take $\sigma$ to be the automorphism of $\Q(\zeta_8)=\Q(\sqrt{2},i)$ which sends $\sqrt{2}\mapsto-\sqrt{2}$, $i\mapsto-i$. Then $f^\sigma=f\otimes \phi_{even}$ as above. From this, we can deduce that $\sigma(\alpha')=8/\alpha'$, and by the same argument, it follows that $\alpha'$ has norm 64. We can take $2+2i$ again as the generator of the unique ideal of norm 64 in $\Q(\zeta_8)$, since again, 2 is ramified. A fundamental unit is $u:=\zeta_8^2+\zeta_8+1$, and take $B=|u|$. Then we can compute $\log_B(|\alpha'/(2+2i)|^8)=8$ and using another numerical calculation, determine that $\alpha'/(2+2i)=-u$. Simplifying, we get $\alpha'=-2 i (2 + \sqrt{2})$, and then by rearranging the resulting equality of $L$-values we arrive at the identity above.\end{proof}

    \begin{Lemma}\label{l4}
 $$F(1/8,3/4))-8\sqrt{-2}F(7/8, 5/4)  = 
\zeta_8^3(1-\sqrt{2}) (4F(5/8, 11/8) + 2\sqrt{-2} F(3/8, 5/8)).$$
\end{Lemma}
\begin{proof}
The proof is identical to the proof of Lemma \ref{l2}.
\end{proof}

The proof of the key lemma in class 8 requires several more steps, because there is no even inner twist of $f_{1/12,11/12}$. We still have an even twist switching the signs of $5$ and $7$, but it is not an inner twist. Crucially, this means that we cannot use the second part of Shimura's result, so the proof will be slightly different. We use on a result of \cite{thomae}, also given in \cite{bailey} and \cite{coxeter} to handle this. We also rely on the following mild irrationality conjecture for this class only.

\begin{Conjecture}\label{con}
    $F(5/12,7/12)/F(7/12,17/12)$ is not an element of $\Q(\sqrt{3})$.
\end{Conjecture}

Proving this conjecture is likely very difficult and is outside the scope of this paper.

\begin{proposition}\label{thomae}
   For any rational  numbers $r$ and $s$, we have that   $$F(r,s)=-2^{6-4s-4r}\frac{\cos(\pi s)}{\cos\pi(r-1/2)} F(3/2-s,3/2-r)-2^{4s-8r+1}\frac{\sin(\pi r)}{\sin\pi(r-1/2)}F(s-r,s).$$
\end{proposition}

\begin{proof}
The identity is derived from a special case of a result stated in \cite{bailey} that uses a normalized ${}_3F_2$ we will denote $\hat{F}(a,b,c,d,e)$ as defined in Beyer et al. \cite{coxeter}. When $a=1/2,b=1/2$, and $d=1$, we will write $\hat{F}(r,s)$. Then Bailey's result specialized to our case says that $$\hat{F}(r,s)=\alpha_1\hat{F}(3/2-s,3/2-r)+\alpha_2\hat{F}({{r}, {1 + r - s}, {r}, {1/2 + r}, {1/2 + r},})$$ for explicit constants $\alpha_1$ and $\alpha_2$ depending on $r$ and $s$. Then we use a different result of Thomae, (see \cite{aar} Corollary 3.3.6), which for the normalized $\hat{F}$ directly tells us $$\hat{F}({{r}, {1 + r - s}, {r}, {1/2 + r}, {1/2 + r}})=\hat{F}({{1/2}, {1/2}, {s-r}, {1 }, {s}}).$$ Rewriting this in terms of $F(r,s)$, we get to the identity.
\end{proof}

The identity becomes much simpler when $s=1-r$.

\begin{corollary}\label{cor}
   If $b$ is the denominator of $r$ and $r<b/2$, we have $$F(r,1-r)=-4F(r+1/2,3/2-r)+2^{5-12r}\tan(\pi r)F(1-2r,1-r).$$
\end{corollary}

We also need the following Lemma.

\begin{Lemma}\label{threeterm}
    For $r=i/b$, $b=4,8,12,24$ and $i$ any positive integer coprime to $b$ and less than $b/2$, we have $$F(r,1-r)/2=F(AM_1(r,1-r))+F(A M_2M_1(r,1-r)).$$ Note that $AM_1(r,1-r)=(2r,r+1/2)$ and $A M_2M_1(r,1-r)=(2r,r+1),$  where $M_1$, $M_2$ are as in \eqref{eq:M&N}.
\end{Lemma} 
\begin{proof} 
    Set $s=1-r$. We use that if $r=i/b$ where $a$ and $b$ are in lowest terms, then $N=2b$. First, note $$A(r,s)=(s-r,r)=(1-2r,r).$$  Also, \eqref{kummer} maps $$K(1-2r,r)=(2r,r+1/2).$$ The Fourier expansions of $\mathbb{K}_2(2r,r+1/2)(b\tau)$ and $\mathbb{K}_2(r,1-r)(2b\tau)$ both have initial term $q$, and they are in the same Hecke orbit. We can then check in each case that $\mathbb{K}_2(2r,r+1/2)(b\tau)$ and $\mathbb{K}_2(2r,r+1)(b\tau)$  are Kummer twists of each other, and their sum is equal to $\mathbb{K}_2(r,1-r)(2b\tau)$. Also, $$F(r,s)=-2\pi i\int_0^{i\infty}\mathbb{K}_2(r,s)(2b\tau)d\tau.$$ If we change the variable of $\mathbb{K}_2(2r,r+1/2)(b\tau)$ to $\mathbb{K}_2(2r,r+1/2)(2b\tau)$, that multiplies the integral by $2$. Multiplying through by $-2\pi i$ and integrating both sides of the above equality gives \begin{align*}
        F(r,1-r)&=-2\pi i\int_{0}^{i\infty} \mathbb{K}_2(r,1-r)(2b\tau)d\tau\\&
        =-2\pi i\int_0^{i\infty}\frac{\mathbb{K}_2(2r,r+1/2)+\mathbb{K}_2(2r,r+1)}{2}(2b\tau)d(2\tau)\\&=2(F(2r,r+1/2)+F(2r,r+1)).
    \end{align*} 
\end{proof}
 Numerical evidence suggests Lemma \ref{threeterm} holds for other positive integers $b$ in addition to the ones listed, but a new proof would be needed. 

 \smallskip
\begin{remark}
    There are various other versions of these identities arising from other families. They are easily computed using the Fourier coefficients.
\end{remark} 

With Corollary \ref{cor} and Lemma \ref{threeterm}, we are able to prove the key lemma of this section, with the assumption of Conjecture \ref{conj}.

\begin{Lemma}\label{l5}
     $$F(1/12, 11/12) + 16 i F(11/12, 13/12) = 
 \zeta_3^2 (-4 F(5/12, 7/12) - 4 i F(7/12, 17/12))$$
\end{Lemma}

\begin{proof}
     The even twist switching $5$ and $7$ in this case is $\phi_{12}$ has the label 12.b in the LMFDB, and its Gauss sum is $\sqrt{12}$. This means that $f_{r,s}$ and $f_{r,s}\otimes \phi_{12}$ have different LMFDB labels, as noted in Remark \ref{lmfd}. As before, if $L(f_{1/8,7/8},1)=\alpha L(f_{1/8,7/8}\otimes \phi_{12},1)$, we can conclude   $$F(1/12, 11/12) + 16 i F(11/12, 13/12) =  \frac{1-\alpha}{1+\alpha} (-4 F(5/12, 7/12) - 4 i F(7/12, 17/12)).$$ We do know that $\sqrt{12}\alpha$ is in $\Q(i)$, which means $\alpha\in \Q(\sqrt{-3})=\Q(\zeta_3)$.

Since applying the Atkin--Lehner involution to $F(1-2r,1-r)$ gives $F(r,1-r)$, we have already shown three term identities in our cases in Lemma \ref{threeterm}. This means for any specific $r$, this can be rewritten as a 4-term identity. In our case, $r=1/12$, from Lemma \ref{threeterm} we have $$F(5/12,7/12)/2=F(5/6,11/12)+F(5/6,17/12)$$ as well as the similar identity $$2F(11/12,13/12)=F(5/6,11/12)-F(5/6,17/12).$$ Adding these up we get that $$2F(5/6,11/12)=F(5/12,7/12)+4F(11/12,13/12).$$
Combined with Corollary \ref{cor}, this produces the 4-term identity, 

\begin{equation}\label{4term}
    F(1/12,11/12)=-4F(7/12,17/12)+8\tan(\pi/12)[F(5/12,7/12)+4F(11/12,13/12)].
\end{equation}
We know from above that $$F(1/12, 11/12) + 16 i F(11/12, 13/12) =  (a+b\sqrt{-3}) (-4 F(5/12, 7/12) - 4 i F(7/12, 17/12))$$ for some rational $a$ and $b$. Taking the real and imaginary part of both sides, we obtain that $$F(1/12,11/12)=4aF(5/12,7/12)-4b\sqrt{3}F(7/12,17/12)$$ and $$16F(11/12,13/12)=4aF(7/12,17/12)+4b\sqrt{3}F(5/12,7/12).$$ If we put these into the identity (\ref{4term}), we get an equation only involving $F(5/12,7/12)$ and $F(7/12,17/12)$,
\begin{align*}
        &-4aF(5/12,7/12)+4b\sqrt{3}F(7/12,17/12)=-4F(7/12,17/12)+\\&4\tan(\pi/12)[F(5/12,7/12)-16\sqrt{3}aF(7/12,17/12)-16bF(5/12,7/12)].
\end{align*}
This simplifies to $$4 (-2 + \sqrt{3} - a + (-3 + 2 \sqrt{3}) b)F(5/12,7/12)=4 (1 - (-2 + \sqrt{3}) a + \sqrt{3} b)F(7/12,17/12).$$
If both of these constants are nonzero, that implies $F(5/12,7/12)=CF(7/12,17/12)$ for some constant $C$ contained in $\Q(\sqrt{3})$. Assuming Conjecture \ref{con}, this cannot be true, so both of these constants must be zero. This gives us a system of two equations in two variables, which by linear algebra has a unique solution, specifically $a=b=-1/2$. Since $-1/2-\sqrt{-3}/2=\zeta_3^2$, we are done.
\end{proof}

At least, the result for the Kummer twist is similar before.

\begin{Lemma}\label{l6}
    $$ F(1/12,2/3)-16F(11/12,4/3)  = 
(-4 + 6 i + (2 - 4 i) \sqrt{3})[- F( 5/12,4/3)+iF(7/12,2/3) ]$$
\end{Lemma}

\begin{proof}
This is identical to the proof of Lemma \ref{l2}.
\end{proof}

\subsection{Extending Kummer's Formula}
Assuming the Lemmas of the previous section, we can now prove Theorem \ref{relations} via a straightforward by intricate computation. We illustrate with class 6. We omit the details for the other two cases, but the proof is extremely similar.

\begin{theorem}\label{f5}
    $L(f_{1/8,5/8},1)=-\frac{1}{2}\zeta_{48}\cdot L(f_{1/8,1},1).$
\end{theorem}

\begin{proof}[Proof of Theorem \ref{f5}]
    Using Lemma \ref{l1}, compute \begin{align*}
  &  L(f_{1/8,5/8},1)\\&=F(1/8, 5/8) - 2 \sqrt{3}F(3/8, 7/8) - 4 \sqrt{-3} F(5/8, 9/8) - 
 8 i F(7/8, 11/8)\\&
=(F(1/8, 5/8)- 8 i F(7/8, 11/8)) -2 \sqrt{3}(F(3/8, 7/8) + 2i F(5/8, 9/8))\\&
=2(1-\sqrt{-2})(F(3/8,7/8)+2iF(5/8,9/8))-2 \sqrt{3}(F(3/8, 7/8) + 2i F(5/8, 9/8))\\&
=2(1-\sqrt{-2}-\sqrt{3})(F(3/8,7/8)+2iF(5/8,9/8)).
\end{align*} 

Then apply the Kummer transformation to get \begin{align*}
 &   2(1-\sqrt{-2}-\sqrt{3})(F(3/8,7/8)+2iF(5/8,9/8))\\&
=2(1-\sqrt{-2}-\sqrt{3})(2 \cos(\pi/8) F(5/8, 1)+i\cos(\pi/8)F(3/8,1))\\&
=2i\cos(\pi/8)(1-\sqrt{-2}-\sqrt{3})(F(3/8,1)-2iF(5/8,1)).
\end{align*}
We now need to do the first half of this process in reverse using Lemma \ref{l2} to get the desired identity. 
\begin{align*}
  &  L(f_{1/8,1},1)\\&=F(1/8, 1) - 2 \sqrt{-3} F(3/8, 1) - 4 \sqrt{3} F(5/8, 1) + 8 i F(7/8, 1)\\&
=(F(1/8, 1)+ 8 i F(7/8, 1)) -2 \sqrt{-3}(F(3/8, 1) - 2i F(5/8, 1))\\&
=-2(1 - \sqrt{-2}) (1 + \sqrt{2})(F(3/8,1)-2iF(5/8,1))-2 \sqrt{-3}(F(3/8, 1) - 2i F(5/8, 1))\\&
=-2((1 - \sqrt{-2}) (1 + \sqrt{2})+\sqrt{-3}
)(F(3/8,1)-2iF(5/8,1)).
\end{align*} 
Putting the two computations together, we get the relation $$L(f_{1/8,5/8},1)=\frac{2i\cos(\pi/8)(1-\sqrt{-2}-\sqrt{3})}{-2((1 - \sqrt{-2}) (1 + \sqrt{2})+\sqrt{-3})}L(f_{1/8,1},1).$$ Simplifying this constant and dividing to the other side of the equation gives us the desired result.
\end{proof}

%%%%%%%%%%%%%%

\section{Further Directions}
We covered all the non-CM classes except 10 and 11, which have $\varphi(M)=8$. In the examples we covered, the group $G$ generated by $A$ and $K$ acts transitively on the set of conjugates, which enables us to find out three-term identities. When $\varphi(M)=8$, the group action is no longer transitive. For this reason, we anticipate further identities such as the Kummer transformation will be required to understand these cases. We leave this to future work. There are also other hypergeometric data with which our methods could apply. For example, recently the author has found similar relations for certain ${}_2F_1(1)$ series in \cite{rosen}. In addition,  ${}_2F_1(z)$ hypergeometric functions are often modular forms on their underlying monodromy group, and then we can use the integral formula to obtain special $L$-values in the same manner as above. The datum $HD_3:=\{\{1/3,2/3\},\{1,1\}\}$ is the most approachable of these, since  ${}_2F_1(HD_3,z)$ is known to be a weight 1 modular form on $\Gamma_0(3)$, and has a classical expression as an eta quotient (see Borwein and Borwein \cite{piagm}). The corresponding $\mathbb{K}_3$ function is mentioned in \cite{aglt}, and its further properties are studied in a paper of Grove \cite{grove}. More generally, Yang \cite{yang} has shown that modular forms on arithmetic triangle groups can be written as a combination of  ${}_2F_1(z)$ functions. However, since in general there are no cusps and thus no $q$-expansions or eta products, these cases likely would be much harder.

\section{Appendix I: The table of Hecke Eigenforms built from $\mathbb{K}_2$} \label{append}
In the appendix, we fix twists of the completed modular forms we obtain using this method in Table \ref{tab:my_label}. Note 6 of these are already given in \cite{aglt}, but we repeat them here for completeness.  Observe also that changing the $\mathbb{K}_2(r,s)$ in this table to $F(r,s)$ also gives $N$ times the $L$-value of the eigenform, and so the table also gives us the special $L$-values obtained in the Galois families, though in some cases these can be simplified. We omit the self-twisted families, as well as class 2. The latter, as well as class 1, only has a single conjugate for each family provided in the table. Finally, the self-twisted families arise from the Atkin-Lehner involution applied to a family in Table \ref{tab:my_label}. The conjugates may be obtained using \eqref{atkinlehner}, and constructing the eigenform involves simply using $q$-series identities like those in the proof of Lemma \ref{threeterm}.

\begin{table}[ht]
\centering
    \begin{tabular}{c|c}
       Label  & Eigenform \\\hline 
        3.a & $f_{1/3,7/6}=\mathbb{K}_2(1/3,7/6)$+2 $\mathbb{K}_2(2/3,5/6)$ \\\hline
        3.b & $f_{1/3,2/3}=\mathbb{K}_2(1/3,2/3)+2\mathbb{K}_2(2/3,4/3)$\\\hline
         4.a  & $f_{1/8,9/8}=\mathbb{K}_2(1/8,9/8)+2\mathbb{K}_2(3/8,11/8)$\\\hline
        4.b  & $f_{1/8,1/12}=\mathbb{K}_2(1/8,1/2)+2i\mathbb{K}_2(3/8,1/2)$ \\\hline
         5.a    & $f_{1/4,3/4}=\mathbb{K}_2(1/4, 3/4) + 4 i\mathbb{K}_2(3/4, 5/4)$ \\\hline
          5.b    & $f_{1/4,1}=\mathbb{K}_2(1/4, 1) - 4 i\mathbb{K}_2(3/4, 1)$ \\\hline
         6.a  & $f_{1/8,5/8}=\mathbb{K}_2(1/8, 5/8) - 2 \sqrt{3} \mathbb{K}_2(3/8, 7/8) $\\&$- 4 \sqrt{-3} \mathbb{K}_2(5/8, 9/8) - 8 i \mathbb{K}_2(7/8, 11/8)$ \\\hline
        6.b & $f_{1/8,1}=\mathbb{K}_2(1/8, 1) - 2 \sqrt{-3} \mathbb{K}_2(3/8, 1])$\\&$- 4 \sqrt{3} \mathbb{K}_2(5/8, 1) + 8 i \mathbb{K}_2(7/8, 1)$ \\\hline
        7.a  & $f_{1/8,7/8}=\mathbb{K}_2(1/8, 7/8) + 2\sqrt{2} \mathbb{K}_2(3/8, 5/8) $\\&$+ 4i \mathbb{K}_2(5/8, 11/8) + 8\sqrt{-2} \mathbb{K}_2(7/8, 9/8)$ \\\hline
        7.b  & $f_{1/8,3/4}=\mathbb{K}_2(1/8, 3/4) + 2\sqrt{-2} \mathbb{K}_2(3/8, 5/4) $\\&$+ 4 \mathbb{K}_2(5/8, 3/4) - 8\sqrt{-2} \mathbb{K}_2(7/8, 5/4)$\\\hline
         8.a  & $f_{1/12,11/12}=\mathbb{K}_2(1/12, 11/12) - 4 \mathbb{K}_2(5/12, 7/12) $\\&$ - 4 i \mathbb{K}_2(7/12, 17/12)  + 
 16 i \mathbb{K}_2(11/12, 13/12)$ \\\hline
 8.b  & $f_{1/12,2/3}=\mathbb{K}_2(1/12, 2/3) + 4 \mathbb{K}_2(5/12, 4/3) $\\&$+ 4 i \mathbb{K}_2(7/12, 2/12) + 16 i \mathbb{K}_2(11/12, 4/3)$\\\hline
        9.a &  $f_{1/24,23/24}=-16 \sqrt{-2} \mathbb{K}_2(23/24, 25/24) + 8 \sqrt{14} \mathbb{K}_2(19/24, 29/24)$\\& $ +  8 \sqrt{7} \mathbb{K}_2(17/24, 31/24)  - 4 i \mathbb{K}_2(13/24, 35/24) + 4 \sqrt{2} \mathbb{K}_2(11/24, 13/24) $\\&$  + 2  \sqrt{-14} \mathbb{K}_2(7/24, 17/24)   - 2 \sqrt{-7} \mathbb{K}_2(5/24, 19/24) + \mathbb{K}_2(1/24, 23/24)$\\\hline
        9.b  & $f_{1/24,7/12}=16 \sqrt{-2} \mathbb{K}_2(23/24, 17/12) + 8 \sqrt{-14} \mathbb{K}_2(19/24, 13/12)$\\& $ + 8 \sqrt{7} \mathbb{K}_2(17/24, 11/12)  - 4 i \mathbb{K}_2(13/24, 7/12)  + 4 \sqrt{-2} \mathbb{K}_2(11/24, 17/12) $\\& $+ 2  \sqrt{-14} \mathbb{K}_2(7/24, 13/12)  - 2 \sqrt{7} \mathbb{K}_2(5/24, 11/12) + \mathbb{K}_2(1/24, 7/12)$\\\hline
     10.a &   $f_{1/24,17/24}=\mathbb{K}_2(1/24, 17/24) - 2 \sqrt{6} \mathbb{K}_2(5/24, 13/24) - 2 \sqrt{-15} \mathbb{K}_2(7/24, 23/24) $\\&$ - 4 \sqrt{10}\mathbb{K}_2(11/24, 19/24)  +  4 \sqrt{-15} \mathbb{K}_2(13/24, 29/24) - 8 \sqrt{10} \mathbb{K}_2(17/24, 25/24) $\\&$ - 
 8 \mathbb{K}_2(19/24, 35/24)  - 16 \sqrt{-6} \mathbb{K}_2(23/24, 31/24)$\\\hline
10.b &  $f_{1/24,5/6}=\mathbb{K}_2(1/24, 5/6) - 2 \sqrt{6} \mathbb{K}_2(5/24, 7/6) + 2 \sqrt{-15} \mathbb{K}_2(7/24, 5/6)$\\&$ + 4 \sqrt{-10} \mathbb{K}_2(11/24, 7/6)  + 4 \sqrt{15} \mathbb{K}_2(13/24, 5/6) + 8 \sqrt{10}
   \mathbb{K}_2(17/24, 7/6) $\\&$+8 i \mathbb{K}_2(19/24, 5/6) -16 \sqrt{-6} \mathbb{K}_2(23/24, 
   7/6)$\\\hline
  
    \end{tabular}
    \caption{Choices of Hecke Eigenforms constructed from Galois families. This is a complete list of such Hecke eigenforms up to quadratic twisting. We fix the twist given in this table throughout the paper.}
    \label{tab:my_label}
\end{table}
Note for 3.a, we actually get an oldform of level 72, so the coefficients that are divisible by 2 are not consistent with the LMFDB, but after applying a Hecke operator $T_p$ with $p$ coprime to the level, they become the same.  

 \section{Appendix II: Geometric Considerations}

The definition of Galois and of conjugates used in this paper, and also in \cite{aglt}, is somewhat ad hoc. In the appendix, we provide a geometric interpretation for these definitions. Essentially, the set of conjugates is in bijection with the elements of the de Rham realization for a hypergeometric motive.

 The base varieties for hypergeometric motives with parameter set of length 3 are surfaces. For nondegenerate cases, the varieties can be defined as in \cite{deinesetal2},\cite{flrst}, \cite{kellyvoight}, \cite{robertsvillegas} as $$y^M=Y_{r,s}:=x_1^{M/2}(1-x_1)^{M/2}x_2^{M(1-r)}(1-x_2)^{M(s-r)}(1-x_1x_2)^{M/2},$$ assuming that $r$ and $s-r$ are between 0 and 1. Then by the integral formula for hypergeometric series, we have $$\int_0^1\int_0^1\frac{dx_1dx_2}{y}=\int_0^1\int_0^1\frac{dx_1dx_2}{\sqrt[M]{Y_{r,s}}}=P(r,s).$$ The variety is badly singular for the locus $(0,x_1,1/x_1)$, i.e. when $x_1x_2=1$. In our setting, these are surfaces over $\C$, so there is a minimal resolution of the singularities, which we denote by $X_{r,s}$. For higher length data, an explicit partial compactification for these surfaces in the generic case $\lambda \neq 0,1$ is introduced by Kelly and Voight \cite{kellyvoight}, which provides a realization of higher weight hypergeometric motives. 

 In this paper, we are only concerned with the de Rham realization of the hypergeometric motive, which occurs as a subspace of in $H^2_{dR}(X_{r,s},\C)$. For data of length 2 as in \cite{archinard}, the correct subspace can be found using the action of $\mu_M$, the $M$th roots of unity, on the space $H^{2,0}(X_{r,s})$. We describe this process below, which is adequate for our purposes. However, we will see that for general data of length 3, this idea does not necessarily make sense.
 
 Throughout, write $\zeta_n=e^{2\pi i/n}$. Parallel to the cases in e.g. \cite{archinard}, the action is given explicitly by $(x_1,x_2,y)\mapsto (x_1,x_2,\zeta_M^{-1}y)$ on $y^M=\alpha_{r,s}$. This induces an action on $H^{2}_{dR}(X_{r,s},\C)$, and decomposes the cohomology into eigenspaces: $$H^{2}_{dR}(X_{r,s},\C)\otimes_{\Z}\Q=\bigoplus_{n=0}^{M-1}V_n.$$   By Hodge theory, the de Rham cohomology decomposes as $$H^2_{dR}(X_{r,s},\C)\cong H^{2,0}(X_{r,s})\oplus H^{1,1}(X_{r,s})\oplus H^{0,2}(X_{r,s}),$$ where $H^{p,q}(X_{r,s})$ denotes the Dolbeault cohomology. This induces a decomposition $V_n\cong V_{n}^{2,0}\oplus V_n^{1,1}\oplus V_{n}^{0,2}$.
 
 Recall by Dolbeault's theorem, $H^{2,0}(X_{r,s})$ is isomorphic to the space of holomorphic differential 2-forms. We construct differentials in $H^{2,0}(X_{r,s})$ following the procedure outlined in Archinard \cite{archinard}, Deines et al. \cite{glc}, and  Wolfart \cite{wolfart} for generalized Legendre curves. These differentials are constructed on $y^M=\alpha_{r,s}$ as $$\omega_{n,\mathfrak{a}}=\frac{x_1^{a_1}(1-x_1)^{a_2}x_2^{a_3}(1-x_2)^{a_4}(1-x_1x_2)^{a_5}\cdot dx_1dx_2}{y^n},$$ where $\mathfrak{a}=(a_1,a_2,a_3,a_4,a_5)$. We then obtain a differential on $X_{r,s}$ by pulling back along the desingularization map - by an abuse of notation, we write $\omega_{n,\mathfrak{a}}$ for a differential on $X_{r,s}$ as well. The pullback of these differentials is not generally holomorphic, except for certain choices of $\mathfrak{a}$. For generalized Legendre curves, Archinard \cite{archinard} writes down the conditions on $\mathfrak{a}$ explicitly using the desingularization map. In our setting, the condition that $(r,s)\in \mathbb S_2'$ is a sufficient condition for the differential to be holomorphic.

To set this up, recall by Lemma \ref{1/b}, there exists a rational number $s_1$ so that $(1/b,s_1)\in \mathbb S_1$ is conjugate to $(r,s)$. If there are multiple such numbers, we choose $s_1$ to be minimal. Therefore, there exists an integer $c$ coprime to $M$ and unique (up to the choice of $r,s$ and $c$) integers $k$ and $h$ so that $r=c\cdot1/b+k$ and $s=c\cdot s_1+h$. By elementary number theory similar to the proof of Lemma \ref{1/b}, it is possible to choose a unique $c$ so that $-M<c<M$ assuming $r<1$.

\begin{Lemma}\label{geom}
    If $(1/b,s_1)$ and $(r,s)$ are as above and assuming that $r<1$, then the period of $\mathbb{K}_2(r,s)$ is equal the period $\omega_{c,\mathfrak{a}}$ up to $\pi$ times an algebraic number if $c\geq 0$ and is equal to the period of $\omega_{M+c,\mathfrak{a}}$ if $c<0$. That is, conjugate $\mathbb{K}_2(r,s)$ correspond to periods arising from the same hypergeometric motive. Moreover, $\omega_{n,\mathfrak{a}}$ lies in the subspace $V_n$.
\end{Lemma}
\begin{proof}
    Let $y=\sqrt[M]{C(1/b,s_1)}$ and $y'=\sqrt[M]{\alpha_{r,s}}$. By construction, $dx_2dx_2/y'$ is equal to $\mathbb{K}_2(r,s)$ up to an algebraic constant.  To show this is also a differential on $y^M=C(1/b,s_1)$, we need to find a vector $\mathfrak{a}=(a_1,a_2,a_3,a_4,a_5)$ so that \begin{equation}\label{diff}
        \frac{dx_1dx_2}{y'}=\frac{x_1^{a_1}(1-x_1)^{a_2}x_2^{a_3}(1-x_2)^{a_4}(1-x_1x_2)^{a_5}dx_1dx_2}{y^c}. 
    \end{equation}By substituting $r=c\cdot 1/b+k$ and $s=c\cdot s_1+h$ and manipulating the differential, we find that the above is true for $$\mathfrak{a}=\left(\frac{c-1}{2},\frac{c-1}{2},c h - c k + (1/b - s_1) (1 - c^2),1 + c - c k - \frac{1-c^2}{b} ,\frac{c-1}{2}\right).$$ We know $c,h$ and $k$ are integers already. Since $M$ is even, $(c-1)/2$ is always an integer, as $c$ is coprime to $M$. Moreover, it is easy to check that $M\mid 1-c^2$ for all $c$ coprime to $M$ for $M=2,4,6,8,12,24$, the only possibilities for $M$ in our setting. Therefore, the entries of $\mathfrak{a}$ are integers. When $c\geq 0$, we are done. If $c\leq 0$, note that $y^M$ is a polynomial by definition. Multiplying the top and bottom of the right hand side in equation \eqref{diff} and adjusting the $a_i$ using $y^M$ gives a differential of the form $\omega_{M+c,\mathfrak{a}}$ as desired.

    The last claim is immediate from the definition of $\omega_{n,\mathfrak{a}}$, as clearly multiplying $y$ by $\zeta_M^{-1}$ produces the eigenvalue $\zeta_M^n$.
\end{proof}

\begin{remark}
    The condition $M\mid 1-c^2$ does not hold for an arbitrary choice of $M$ and $c$. Hence, our construction of conjugates does not work in general. For instance, even for $M=16$, we could take $c=3$, but clearly 16 does not divide $3^2-1=8$. The underlying issue is that the definition of conjugates was essentially lifted from the ${}_2F_1(\lambda)$ setting, where Lemma \ref{geom} follows from a classical result about Riemann surfaces (refer to e.g. \cite{archinard}). To our knowledge, an explicit description of the complete space of holomorphic differentials and their periods on general hypergeometric surfaces remains an open problem.
\end{remark}

There are no examples where $r=a/b>1$ in the non-CM examples. However, these cases are also easily handled with a slight modification.

\begin{theorem}
   A non-CM family is Galois if and only if for all $n$ satisfying $(n,M)=1$,  $V_n^{2,0}$ contains exactly one differential of the shape $\omega_{n,\mathfrak{a}}$. Moreover, for a Galois family, $\dim_{\Q}V_n^{2,0}=1$ and is generated by $\omega_{n,\mathfrak{a}}$.
\end{theorem}
\begin{proof}
The converse is straightforward - the one holomorphic differential, $\omega_{i,\mathfrak{a}}$, in $V_{i}^{2,0}$, has period $P(i/b,s_i)$ for some $s_i$ by the definition and the integral formula for hypergeometric series. Each of these corresponds to a modular form $\mathbb{K}_2(i/b,s_i)$. If $\mathbb{K}_2(i'/b,s_{i'})$ is conjugate to $\mathbb{K}_2(i/b,s_i)$, then it corresponds to some differential $\omega_{i',\mathfrak{a}}$ by the discussion above. But there is only one such differential in each $V_n$, and so we must have $i\neq i'$ for all conjugates. This implies the family is Galois.

Now assume $\{\mathbb{K}_2(i/b,s_i)\}_{i\in (\Z/b\Z)^\times}$ is a Galois family. Using Table \ref{tab}, we can easily check that for all such families, $s_i>1/2$. Then by the Lemma above, this implies that each pair $(i/b,s_i)$ corresponds to a differential $\omega_{i,\mathfrak{a}}$. But by the definition of Galois, each of the $i$ is distinct, and so there is exactly one $\omega_{i,\mathfrak{a}}$ in the space $V_i$. A result of Fedorov \cite{fedorov} allows us to compute the Hodge numbers for hypergeometric motives over $\Q(\zeta_M)$ for $\lambda\neq 1$. In all of the Galois cases, these Hodge numbers are $(1,1,1)$, since $s_i>1/2$. When moving to $\lambda=1$, the rank is reduced by one in the middle cohomology, making the Hodge numbers $(1,0,1)$. Due to the correspondence above, we can assume that $\omega_{1,\mathfrak{a}}$ is in the cohomology associated to the hypergeometric motive, and so we deduce that $\dim_{\Q(\zeta_M)}H_{r,s}^{2,0}=1$. Therefore, the dimension over $\Q$ is precisely $\varphi(M)$, corresponding to the $\varphi(M)$ holomorphic conjugates. There is exactly one conjugate in each $V_{n}^{2,0}$ by above, and there are precisely $\varphi(M)$ subspaces $V_{n}^{2,0}$. This forces $\dim V_n^{2,0}=1$.
\end{proof}

\begin{remark}
    By the same reasoning as the proof above, we can conclude that $\dim_{\Q}H_{r,s}=2\varphi(M)$. In the special case of class 6, this means $H_{r,s}$ has Hodge numbers $(2,0,2)$. This Hodge structure admits endomorphisms by $\Q(\zeta_4)=\Q(i)$ for the reason described above, and so a result of Totaro \cite{totaro} implies that the endomorphism algebra of the Hodge structure is isomorphic to a quaternion algebra over $\Q$. The remaining non-CM cases also have extra structure like this.
\end{remark}

\bibliographystyle{plain} 
\bibliography{references}
\end{document}